\documentclass[reqno]{amsart}
\usepackage{graphicx}
\usepackage{amsmath}
\usepackage{amsthm}
\usepackage{amsfonts}
\usepackage{amssymb}
\usepackage{mathrsfs}
\usepackage{enumerate}
\usepackage[usenames,dvipsnames]{xcolor}
\usepackage{tikz,caption}
\usetikzlibrary{arrows,shapes,trees}
\usepackage{verbatim, hyperref}
\hypersetup{linkcolor=blue, citecolor=red, colorlinks=true}
\usepackage{xcolor}
\definecolor{lightgray}{gray}{0.75}
\providecommand{\mathcol}[1]{\text{\colorbox{lightgray}{$#1$}}}
\usepackage{array}
\usepackage{mathtools}  
\usepackage{cite}          
\allowdisplaybreaks
\usepackage[a4paper,margin=1in,footskip=0.25in]{geometry}

\newtheorem{theorem}{Theorem}
\newtheorem{lemma}[theorem]{Lemma}
\theoremstyle{definition}
\newtheorem{definition}[theorem]{Definition}
\newtheorem{example}[theorem]{Example}
\newtheorem{remark}[theorem]{Remark}
\newtheorem{problem}{Problem}
\newcommand{\beq}{\begin{equation}}
\newcommand{\eeq}{\end{equation}}
\newcommand{\barr}[1]{\begin{array}{#1}}
\newcommand{\earr}{\end{array}}
\providecommand{\lp}[0]{\mathscr{L}\!\!-\!\!\mathscr{P}}

\providecommand{\Z}[0]{\mathbb{Z}}
\providecommand{\R}[0]{\mathbb{R}}

\providecommand{\N}[0]{\mathbb{N}}

\newcommand{\lr}[1]{\left(#1\right)}

\begin{document}

\title{Operator~Diagonalizations of~Multiplier~Sequences}


\author{Robert D. Bates}


%\institute{Robert D. Bates \at
%					 University of Hawai'i at M\=anoa \\
%					 Department of Mathematics \\
%					 Honolulu Community College \\
%					 874 Dillingham Blvd (BLD7-410) \\ 
%					 Honolulu, HI 96817 \\
%					 (808)847-9863
%           \email{rdbates@math.hawaii.edu}           %  \\
%          \emph{Present address:} of F. Author  %  if needed
%          \and
%          S. Author \at
%          second address
%}

\date{\today}

\maketitle

\begin{abstract}
We consider hyperbolicity preserving operators with respect to a new linear operator representation on $\R[x]$. In essence, we demonstrate that every Hermite and Laguerre multiplier sequence can be diagonalized into a sum of hyperbolicity preserving operators, where each of the summands forms a classical multiplier sequence. Interestingly, this does not work for other orthogonal bases; for example, this property fails for the Legendre basis. We establish many new formulas concerning the $Q_k$'s of Peetre's 1959 differential representation for linear operators in the specific case of Hermite and Laguerre diagonal differential operators. Additionally, we provide a new algebraic characterization of the Hermite multiplier sequences and also extend a recent result of T. Forg\'acs and A. Piotrowski on hyperbolicity properties of the polynomial coefficients in hyperbolicity preserving Hermite diagonal differential operators. 
%\keywords{differential operators \and hyperbolicity preservers \and multiplier sequences \and orthogonal polynomials}
%\subclass{26C10, 33C52}
\end{abstract}

\section{Introduction}

Define the Jacobi-Theta function by, 
\beq
\Phi(t):=\sum_{n=1}^\infty (2n^4\pi^2e^{9t}-3n^2\pi e^{5t})e^{-n^2\pi e^{4t}}. 
\eeq
It is well known that the Riemann Hypothesis \cite[(1859)]{Rie59} is equivalent to the statement that the integral cosine transform of the Jacobi-Theta function, 
\beq\label{eq:jtfunction}
\int \Phi(t) \cos(xt) dt, 
\eeq
can be uniformly approximable by polynomials with only real zeros (see for example G. Csordas, T. Norfolk, and R. Varga \cite{CNV86}) (see also \cite{CC13,CV88,CV90,Cso98}). In 1913, J. Jensen \cite{Jen13} showed that every entire function, 
\beq\label{eq:f(x)}
f(x):=\sum_{k=0}^\infty \frac{\gamma_k}{k!}x^k, 
\eeq
can be uniformly approximated by polynomials with only real zeros if and only if $g_n(x)$ has only real zeros for each $n\in\N_0$, where
\beq
g_n(x):=\sum_{k=0}^n \binom{n}{k} \gamma_k x^k. 
\eeq
Hence, the associated Jensen polynomials, $\{g_n(x)\}_{n=0}^\infty$, have received a great deal of attention in modern times (see for example \cite{DC09,CC83,CC89,CW75}). In particular, M. Chasse in 2011 showed remarkably that the first $2\cdot 10^{17}$ Jensen polynomials of \eqref{eq:jtfunction} have only real zeros \cite[Theorem 177, p. 87]{Cha11}. 

In 1914, G. P\'olya and J. Schur \cite{PS14} gave a complete characterization of hyperbolicity preserving operators (operators that map polynomials with only real zeros to those of the same kind, see Definition \ref{def:hyper}) of the form,
\beq\label{eq:psintro}
T[x^n]:=\gamma_n x^n,\ \{\gamma_n\}_{n=0}^\infty\subset\R,
\eeq
by showing that $f(x)$ (from \eqref{eq:f(x)}) must be uniformly approximable by polynomials with zeros of one sign  (throughout the literature, $\{\gamma_n\}_{n=0}^\infty$ is called a multiplier sequence). Their work was greatly extended in 2009 by J. Borcea and P. Br\"and\'en \cite{BB09} who demonstrated that essentially every linear operator written in J. Peetre's \cite[(1959)]{Pee59} differential operator form, 
\beq\label{eq:do}
T:=\sum_{k=0}^\infty Q_k(x)D^k,\ D:=\frac{d}{dx}, 
\eeq
is hyperbolicity preserving if and only if either 
\beq
T[e^{xw}]:=e^{xw}\sum_{k=0}^\infty Q_k(x)w^k\ \ \text{ or }\ \ T[e^{-xw}]:=e^{-xw}\sum_{k=0}^\infty Q_k(x)(-w)^k
\eeq
is uniformly approximable by real two-variable stable polynomials. 

It was shown by Laguerre in 1882 \cite{Lag82} (later generalized by G. P\'olya \cite[(1913)]{Pol13}, \cite[(1915)]{Pol15}) that every real entire function, which can be uniformly approximated by polynomials with only real zeros, must be of the form, 
\beq
f(x):=cx^m e^{-ax^2+bx} \prod_{k=1}^\omega \left(1+\frac{x}{x^k}\right)e^{-x/x_k}, 
\eeq
where $0\le \omega\le \infty$, $a,b,c\in\R$, $a\ge 0$, $m\in\N_0$, $\{x_k\}_{k=1}^\omega\subset\R$, $x_k\not=0$, and $\sum_{k=1}^\omega \frac{1}{x_k^2}<\infty$. Likewise, real entire functions, that can be uniformly approximated by polynomials with non-positive zeros, must be of the form, 
\beq\label{eq:lpplusprod}
f(x):=cx^me^{bx}\prod_{k=1}^\omega \left(1+\frac{x}{x^k}\right), 
\eeq
where $0\le\omega\le\infty$, $b,c\in\R$, $b\ge 0$, $m\in\N_0$, $\{x^k\}_{k=1}^\omega$, $x_k>0$, and $\sum_{k=1}^\omega \frac{1}{x^k}<\infty$. In 1983, T. Craven and G. Csordas demonstrated an important subclass of functions from \eqref{eq:lpplusprod}, showing that $b\ge 1$ if and only if $f^{(k)}(0)\le f^{(k+1)}(0)$ ($\gamma_k\le \gamma_{k+1}$) for every $k\in\N_0$ \cite{CC83}. This motivated re-investigating some of P.~Tur\'an's results \cite[(1954)]{Tur54} by modifying G. P\'olya and J. Schur's operator, equation \eqref{eq:psintro}, replacing $x^n$ with the $n^{\text{th}}$ Hermite polynomial (see \cite{BC01}). In 2007, A. Piotrowski gave a complete characterization of hyperbolicity preserving operators that diagonalize on the Hermite basis, 
\beq
T[H_n(x)]:=\gamma_n H_n(x),\ \{\gamma_n\}_{n=0}^\infty\subset\R,
\eeq
where for each $n\in\N_0$, $H_n(x)$ denotes the $n^{\text{th}}$ Hermite polynomial. It was demonstrated that $T$ is hyperbolicity preserving if and only if $\{\gamma_n\}_{n=0}^\infty$ is an increasing classical multiplier sequence (from \eqref{eq:psintro}) (see Theorem \ref{thm:msherm}). Recently, there has been significant motivation in characterizing multiplier sequences of any basis (see \cite{FHMS13, Bat14a, BY13, BDFU12, Cha11, FP13a, FP13b, Pio07, Yos13}). In particular, it has become increasingly apparent, the role that orthogonal polynomials seem to play in defining hyperbolicity preserving operators (see also the recent characterization of Laguerre multiplier sequences by P. Br\"and\'en and E. Ottergren \cite{BO12}, see Theorem \ref{thm:mslag}). 

In this paper, we modify J. Peetre's differential representation \cite{Pee59}, giving a new differential representation for study with respect to hyperbolicity preservation (Theorem \ref{thm:classic} and \ref{thm:classic2}). We use this to essentially show that every Hermite and Laguerre multiplier sequence can be written as a sum of classical multiplier sequences (Theorem \ref{thm:herm} and \ref{thm:lag}). Interestingly, the Legendre basis does not enjoy this property (Example \ref{ex:leg}). New methods of determining the differential representation of Hermite and Laguerre diagonal differential operators are found (Theorem \ref{thm:hermqk}, \ref{thm:hermcomplex}, and \ref{thm:lagqk}). Additionally, we give a new algebraic characterization of Hermite multiplier sequences (Theorem \ref{thm:minimal}) and generalize a recent statement of T. Forg\'acs and A. Piotrowski \cite{FP13b}, on the hyperbolicity properties of the $Q_k$'s in \eqref{eq:do} that arise from a Hermite diagonal differential operator (Theorem \ref{thm:hermrealzeros}).  

\begin{definition}\label{def:poly}
We will denote the \textit{Hermite}, \textit{Laguerre}, and \textit{Legendre} polynomials as, $\{H_n(x)\}_{n=0}^\infty$, $\{L_n(x)\}_{n=0}^\infty$, and $\{P_n(x)\}_{n=0}^\infty$, respectively \cite[pp. 157, 187, 201]{Rai60}. For each $n\in\N_0$, these polynomials are given by the following formulas, 
\begin{align}
H_n(x)&=\sum_{k=0}^{[n/2]}\frac{(-1)^k n! 2^{n-2k}}{k!(n-2k)!} x^{n-2k}, \\
L_n(x)&=\sum_{k=0}^n \frac{(-1)^k}{k!}\binom{n}{k}x^k,\ \text{and} \\
P_n(x)&=\sum_{k=0}^{[n/2]} \frac{(-1)^k}{2^n}\binom{n}{k}\binom{2n-2k}{n}x^{n-2k}.
\end{align}
It is well know that these polynomials satisfy the following differential equations \cite[pp. 173, 188, 204, 258]{Rai60}, 
\begin{align}
&\left((-1/2)D^2+(x)D\right)H_n(x)=(n)H_n(x), \\
&\left((-x)D^2+(x-1)D\right)L_n(x)=(n)L_n(x),\ \text{and} \\
&\left((x^2-1)D^2+(2x)D\right)P_n(x)=(n^2+n)P_n(x),  
\end{align}
where $D:=\frac{d}{dx}$. 
\end{definition}

\begin{definition}
Suppose $f(x)$ is an entire function, 
\beq
f(x):=\sum_{k=0}^\infty \frac{\gamma_k}{k!}x^k. 
\eeq
For each $n\in\N_0$, we define the $n^\text{th}$ \textit{Jensen polynomial} associated to the entire function $f(x)$ (or associated to the sequence $\{\gamma_k\}_{k=0}^\infty$) by, 
\beq
g_n(x):=\sum_{k=0}^n \binom{n}{k}\gamma_k x^k. 
\eeq
Likewise, for each $n\in\N_0$, we define the $n^\text{th}$ \textit{reversed Jensen polynomial} by, 
\beq
g_n^*(x):=\sum_{k=0}^n \binom{n}{k} \gamma_k x^{n-k}. 
\eeq
\end{definition}

\begin{definition}\label{def:diagonalize}
Let $T:\R[x]\to\R[x]$ be a linear operator such that $T[B_n(x)]=\gamma_nB_n(x)$ for every $n\in\N_0$, where $\{\gamma_n\}_{n=0}^\infty$ is a sequence of real numbers and $\{B_n(x)\}_{n=0}^\infty$, $\deg(B_n(x))=n$, $B_0\not\equiv 0$, is a basis of real polynomials. Then $T$ will be referred to as a \textit{diagonal differential operator} with respect to the eigenvector sequence, $\{B_n(x)\}_{n=0}^\infty$, and eigenvalue sequence, $\{\gamma_n\}_{n=0}^\infty$. If $\{B_n(x)\}_{n=0}^\infty=\{x^n\}_{n=0}^\infty$ then $T$ is said to be a \textit{classical diagonal differential operator}. Similarly, if $\{B_n(x)\}_{n=0}^\infty$ is the Hermite, Laguerre, or Legendre polynomials (Definition \ref{def:poly}), then $T$ is said to be a \textit{Hermite diagonal differential operator}, \textit{Laguerre diagonal differential operator}, or a \textit{Legendre diagonal differential operator}, respectively. 
\end{definition}

\begin{definition}\label{def:hyper}
Let $T:\R[x]\to\R[x]$ be a linear operator. Operator $T$ is said to be \textit{hyperbolicity preserving} if $T[p(x)]$ has only real zeros whenever $p(x)\in\R[x]$ has only real zeros. If in addition, $T$ diagonalizes on $\{B_n(x)\}_{n=0}^\infty=\{x^n\}_{n=0}^\infty$, $\{B_n(x)\}_{n=0}^\infty=\{H_n(x)\}_{n=0}^\infty$, $\{B_n(x)\}_{n=0}^\infty=\{L_n(x)\}_{n=0}^\infty$, or $\{B_n(x)\}_{n=0}^\infty=\{P_n(x)\}_{n=0}^\infty$, as in $T[B_n(x)]=\gamma_n B_n(x)$ for some sequence of real numbers, $\{\gamma_n\}_{n=0}^\infty$, then $\{\gamma_n\}_{n=0}^\infty$ is called a \textit{classical multiplier sequence}, \textit{Hermite multiplier sequence}, \textit{Laguerre multiplier sequence}, or \textit{Legendre multiplier sequence}, respectively. 
\end{definition}

\begin{definition}
Suppose $T$ is a hyperbolicity preserving operator that diagonalizes on $\{B_n(x)\}_{n=0}^\infty$ and $\{\gamma_n\}_{n=0}^\infty$, where
\beq
\{\gamma_n\}_{n=0}^\infty := \{0,0,\ldots,0,0,\alpha,\beta,0,0,0,\ldots\},\ \alpha,\beta\in\R. 
\eeq
Then, $\{\gamma_n\}_{n=0}^\infty$ is called a \textit{trivial multiplier sequence}. In Theorems \ref{thm:msherm} and \ref{thm:mslag} we will exclude all trivial multiplier sequences. 
\end{definition}

\begin{definition}\label{def:lp}
The \textit{Laguerre-P\'olya class}, denoted as $\lp$, is the set of entire functions that are uniform limits of \textit{hyperbolic polynomials}, real valued polynomials with only real zeros. We define $\lp^s$ to be the entire functions in $\lp$ with Taylor coefficients of the same sign. Likewise, we define $\lp^a$ to be the entire functions in $\lp$ with alternating Taylor coefficients. The notation, $\lp^{sa}$, is defined as $\lp^{sa}:=\lp^s\cup\lp^a$. Given an interval, $I\subseteq\R$, $\lp^* I$ will denote functions in $\lp^*$ that have zeros only in $I$, where $\lp^*$ is either $\lp$, $\lp^s$, $\lp^a$, or $\lp^{sa}$.  
\end{definition}

\begin{theorem}[{T. Craven and G. Csordas \cite[(1983)]{CC83}}]\label{thm:lpinc}
Suppose $f(x)\in\lp^s$. Then $|f^{(k)}(0)|\le |f^{(k+1)}(0)|$ for all $k\in\N_0$ if and only if $e^{-x}f(x)\in\lp^s$. 
\end{theorem}

\begin{remark}\label{rm:lpinc}
In the sequel we will make us of the fact that many of the classes defined above are closed under differentiation. Consider an entire function, 
\beq
f(x):=\sum_{k=0}^\infty \frac{\gamma_k}{k!}x^k.
\eeq
If $f(x)$ is in $\lp^s$, $\lp^a$, or $\lp$, then for every $n\in\N_0$, $f^{(n)}(x)$ is also in $\lp^s$, $\lp^a$, or $\lp$, respectively. Similarly, a slight extension of Theorem \ref{thm:lpinc} shows that if $e^{-\sigma x}f(x)\in\lp^s$ ($\sigma>0$), then for every $n\in\N_0$, $e^{-\sigma x}f^{(n)}(x)\in\lp^s$. Likewise, if $e^{\sigma x}f(x)\in\lp^a$ ($\sigma>0$), then for every $n\in\N_0$, $e^{\sigma x} f^{(n)}(x)\in\lp^a$.  
\end{remark}

\begin{theorem}[{\hspace{1sp}\cite{Bat14a}, \cite{Pee59}, \cite[Proposition 29, p. 32]{Pio07}}]\label{thm:piotr}
If $T:\R[x]\to\R[x]$ is any linear operator, then there is a unique sequence of real polynomials, $\{Q_k(x)\}_{k=0}^\infty\subset\R[x]$, such that 
\begin{equation} \label{eq:linear-operator}
T=\sum_{k=0}^\infty Q_k(x) D^k,\ \mbox{where } D:=\frac{d}{dx}.
\end{equation}
Furthermore, given any sequence of polynomials, $\{B_n(x)\}_{n=0}^\infty$ \textup{(}$\deg(B_n(x))=n$ for each $n\in\N_0$, $B_0(x)\not\equiv 0$\textup{)}, then for each $n\in\N_0$, 
\beq\label{equ13}
Q_n(x)=\frac{1}{B_n^{(n)}}\left(T[B_n(x)]-\sum_{k=0}^{n-1} Q_k(x)B_n^{(k)}(x)\right). 
\eeq
\end{theorem}

\begin{theorem}[{G. P\'olya and J. Schur \cite[(1914)]{PS14}}]\label{thm:ms}
Let $\{\gamma_k\}_{k=0}^\infty$ be a sequence of real numbers. Sequence $\{\gamma_k\}_{k=0}^\infty$ is a positive or negative multiplier sequence if and only if 
\beq
\sum_{k=0}^\infty \frac{\gamma_k}{k!} x^k \in \lp^s.
\eeq
Sequence $\{\gamma_k\}_{k=0}^\infty$ is an alternating multiplier sequence if and only if 
\beq
\sum_{k=0}^\infty \frac{\gamma_k}{k!} x^k \in \lp^a.
\eeq
\end{theorem}

\begin{theorem}[{A. Piotrowski \cite[Theorem 152, p. 140 (2007)]{Pio07}}]\label{thm:msherm}
Let $\{\gamma_k\}_{k=0}^\infty$ be a sequence of real numbers and let $\{g_k^*(x)\}_{k=0}^\infty$ be the sequence of reversed Jensen polynomials associated with $\{\gamma_k\}_{k=0}^\infty$. Sequence $\{\gamma_k\}_{k=0}^\infty$ is a non-trivial positive or negative Hermite multiplier sequence if and only if 
\beq
e^{-x} \sum_{k=0}^\infty \frac{\gamma_k}{k!} x^k = \sum_{k=0}^\infty \frac{g_k^*(-1)}{k!}x^k \in \lp^s.
\eeq 
Sequence $\{\gamma_k\}_{k=0}^\infty$ is a non-trivial alternating Hermite multiplier sequence if and only if 
\beq
e^x \sum_{k=0}^\infty \frac{\gamma_k}{k!} x^k = e^{2x} \sum_{k=0}^\infty \frac{g_k^*(-1)}{k!}x^k \in \lp^a.
\eeq
\end{theorem}

\begin{theorem}[{P. Br\"and\'en and E. Ottergren \cite[(2014)]{BO12}}]\label{thm:mslag}
Let $\{\gamma_k\}_{k=0}^\infty$ be a sequence of real numbers and let $\{g_k^*(x)\}_{k=0}^\infty$ be the reversed Jensen polynomials associated with $\{\gamma_k\}_{k=0}^\infty$. Sequence $\{\gamma_k\}_{k=0}^\infty$ is a non-trivial positive or negative Laguerre multiplier sequence if and only if 
\beq
\sum_{k=0}^\infty g_k^*(-1)x^k \in \R[x]\cap\lp^s[-1,0].
\eeq
There are no non-trivial alternating Laguerre multiplier sequences. 
\end{theorem}

From Theorem \ref{thm:ms}, \ref{thm:msherm}, and \ref{thm:mslag}, it is clear that every Laguerre multiplier sequence is a Hermite multiplier sequence, and every Hermite multiplier sequence is a classical multiplier sequence (see also the classification diagram of K. Blakeman, E. Davis, T. Forg\'acs, and K. Urabe \cite{BDFU12}). 

In the literature, it is common to discuss only the non-negative multiplier sequences. However, many of our results establish strong differences between the sequences from $\lp^a$ and the sequences in $\lp^s$ (see for example Theorem \ref{thm:herminter}). Hence, we will take great care to discuss $\lp^s$ sequences separately from $\lp^a$ sequences. The following example demonstrates the strong differences in differential representation from positive eigenvalues versus alternating eigenvalues. 

\begin{example}\label{ex:hermfinin}
Consider the following hyperbolicity preserving Hermite diagonal differential operators (see Theorem \ref{thm:msherm}), 
\beq
T[H_n(x)]:=nH_n(x)\ \ \ \text{and}\ \ \ W[H_n(x)]:=(-1)^nnH_n(x). 
\eeq
Using the recursive formula from Theorem \ref{thm:piotr}, we calculate $T$ and $W$, 
\beq
T=(x)D+\left(-\frac{1}{2}\right)D^2, 
\eeq
and
\beq
W=(-x)D+\left(2x^2-\frac{1}{2}\right)D^2+\left(-2x^3+x\right)D^3+\cdots. 
\eeq
We observe that $T$ is a finite order differential operator, while $W$ is an infinite order differential operator. This observation makes sense when we note that $\{(-1)^n n\}_{n=0}^\infty$ is not interpolatable by a polynomial (see \cite{Bat14a}). 
\end{example}

The sensitivity of the two classes, $\lp^s$\ and $\lp^a$, can also be seen in the following theorem, which holds for sequences arising from $\lp^s$, but not for sequences arising from $\lp^a$. 

\begin{theorem}[{T. Craven and G. Csordas \cite[(1983)]{CC83}}]\label{thm:mscomb}
Let $\{\gamma_k\}_{k=0}^\infty$ be a positive or negative classical multiplier sequence. Then, for each $m\in\N_0$, 
\beq\label{eq:mscomb}
\left\{\sum_{k=0}^n \binom{n}{k} \gamma_{m+k} \right\}_{n=0}^\infty\ \ \ \text{and}\ \ \ \left\{\sum_{k=0}^m \binom{m}{k} \gamma_{n+k} \right\}_{n=0}^\infty, 
\eeq
are also positive or negative classical multiplier sequences, respectively. 
\end{theorem}

\begin{proof}
For the first sequence, using a Cauchy product, we calculate 
\beq
\sum_{n=0}^\infty \left(\sum_{k=0}^n \binom{n}{k} \gamma_{m+k}\right)\frac{x^n}{n!} = e^x D^m \sum_{n=0}^\infty \frac{\gamma_n}{n!}x^n \in \lp^s. 
\eeq
For the second sequence, using two Cauchy products, we calculate 
\beq\nonumber
\sum_{n=0}^\infty \left(\sum_{k=0}^m \binom{m}{k} \gamma_{n+k}\right)\frac{x^n}{n!}  = e^{-x} D^m e^x \sum_{n=0}^\infty \frac{\gamma_n}{n!}x^n\in \lp^s. \eqno\qed
\eeq
\end{proof}

\begin{example}
We show that Theorem \ref{thm:mscomb} does not hold for $\lp^a$. Consider the following function in $\lp^a$, 
\beq
f(x):=\sum_{k=0}^\infty \frac{(-1)^k}{k!} \frac{x^k}{k!},
\eeq
which is obtained by application of the multiplier sequence $\{\frac{(-1)^k}{k!}\}_{k=0}^\infty$ to the function $e^x$. The sequence, 
\beq
\{\gamma_n\}_{n=0}^\infty=\left\{\sum_{k=0}^n \binom{n}{k} \frac{(-1)^k}{k!}\right\}_{n=0}^\infty, 
\eeq
has the form, 
\beq
\{\gamma_n\}_{n=0}^\infty=\left\{1,0,-\frac{1}{2},-\frac{2}{3},-\frac{5}{8},\ldots,\frac{887}{5760},\ldots\right\}. 
\eeq
Hence, $\{\gamma_n\}_{n=0}^\infty$ is not a multiplier sequence, since there is no function in $\lp^s$ or $\lp^a$ with Taylor coefficients that match the signs of $\{\gamma_n\}_{n=0}^\infty$ (see Theorem \ref{thm:ms}). 
\end{example}

For the reader's convenience we provide the following compilation of combinatorial identities that will be used extensively throughout the paper. These types of calculations have already been observed in the proof of Theorem \ref{thm:mscomb}. 

\begin{theorem}[{\hspace{1sp}\cite[p. 49]{Rio79}, \cite[Proposition 33, p. 35]{Pio07}}]\label{thm:comb}
Given a sequence of real numbers, $\{\alpha_k\}_{k=0}^\infty$, for each $n\in\N_0$, define, 
\beq
\beta_n= \sum_{k=0}^n \binom{n}{k} \alpha_k. 
\eeq
Then, for all $n\in\N_0$, 
\beq
\alpha_n = \sum_{k=0}^n \binom{n}{k} \beta_k (-1)^{n-k}. 
\eeq
In particular, we have, 
\beq
e^x \sum_{n=0}^\infty \frac{\alpha_n}{n!} x^n = \sum_{k=0}^\infty \frac{\beta_n}{n!}x^n\ \ \ \text{and}\ \ \ e^{-x} \sum_{k=0}^\infty \frac{\beta_n}{n!}x^n = \sum_{k=0}^\infty \frac{\alpha_n}{n!}x^n.
\eeq
Similarly, if $\{g_k^*(x)\}_{k=0}^\infty$ are the reversed Jensen polynomials associated with $\{\gamma_k\}_{k=0}^\infty$, then for every $n\in\N_0$, 
\beq\label{eq:jenreverse}
\gamma_n=\sum_{k=0}^n \binom{n}{k} g_k^*(-1)\ \ \ \text{and}\ \ \ g_n^*(-1)=\sum_{k=0}^n \binom{n}{k} \gamma_k (-1)^{n-k}. 
\eeq
Likewise, if $\{\gamma_k\}_{k=0}^\infty$ diagonalizes the classical diagonal differential operator, $T$, then 
\beq\label{eq:diagst}
T[x^n]=\left(\sum_{k=0}^\infty \frac{g_k^*(-1)}{k!} x^k D^k \right)x^n = \gamma_n x^n. 
\eeq 
\end{theorem}

\vspace{.25in}
\section{Operator Diagonalizations of Diagonalizable Operators}

Our main objective is to present a new representation of diagonal differential operators (Theorem \ref{thm:classic}). We will only need to assume that $\deg(Q_k(x))\le k$ for each $k\in\N_0$; a property that all diagonal differential operators have, as the recursive formula of Theorem \ref{thm:piotr} shows (see also \cite{Bat14a}). 

\begin{theorem}\label{thm:classic}
Given a linear operator, $T:\R[x]\to\R[x]$, 
\beq
T=\sum_{k=0}^\infty Q_k(x)D^k,
\eeq
where $\deg(Q_k(x))\le k$ for every $k\in\N_0$. Define the family of sequences, 
\beq\label{eq:aform}
\{b_{n,k}\}_{k=0}^\infty:=\left\{\sum_{j=0}^k\binom{k}{j}Q_{j+n}^{(j)}(0)\right\}_{k=0}^\infty,\ \ \ \ \ \ \ \ n\in\N_0. 
\eeq
For each $n\in\N_0$, define the classical diagonal differential operator, 
\beq\label{eq:tn}
T_n[x^k]:=b_{n,k}x^k. 
\eeq
Then, 
\beq\label{eq:diawriting}
T=\sum_{n=0}^\infty T_nD^n. 
\eeq
Furthermore, the representation in \eqref{eq:diawriting} is unique. 
\end{theorem}

\begin{proof}
We are concerning ourselves with operators defined on $\R[x]$, hence convergence discussions are a non-issue. By Theorem \ref{thm:comb}, for every $n\in\N_0$, we know the differential representation of $T_n$, namely,  
\beq\label{eq:formula}
T_n=\sum_{k=0}^\infty \left(\sum_{j=0}^k \binom{k}{j}b_{n,j}(-1)^{k-j}\right)\frac{1}{k!}x^kD^k= \sum_{k=0}^\infty \frac{Q_{k+n}^{(k)}(0)}{k!}x^kD^k. 
\eeq
Note the calculation $\frac{Q_{k+n}^{(k)}(0)}{k!}x^k$ is precisely the $k^{\text{th}}$ term of the polynomial, $Q_{k+n}(x)$. Hence, each summand, in each $T_n$, is one term from some $Q_k(x)$. Furthermore, no two $T_n$'s use the same term in a particular $Q_k(x)$. Finally, because $\deg(Q_k(x))\le k$, we are assured that every term in every $Q_k(x)$ will be present in some $T_n$. The uniqueness follows from the uniqueness of the differential representation in Theorem \ref{thm:piotr}. 
\end{proof}

\begin{example}
Theorem \ref{thm:classic} can be best understood with the aid of a concrete illustrative. Define the differential operator, 
\beq
T:=\underset{Q_2(x)}{\underbrace{(a_2x^2+b_1x+c_0)}}D^2+\underset{Q_1(x)}{\underbrace{(a_1x+b_0)}}D+\underset{Q_0(x)}{\underbrace{(a_0)}}, 
\eeq
where $a_2,a_1,a_0,b_1,b_0,c_0\in\R$. Using Theorem \ref{thm:classic}, we re-write $T$, in terms of $T_n$'s, 
\beq
\begin{array}{rcl}
T &=& \left(\frac{Q_2^{(2)}(0)}{2!}x^2D^2+\frac{Q_1^{(1)}(0)}{1!}x^1D^1+\frac{Q_0^{(0)}(0)}{0!} x^0D^0\right) D^0\ + 
\\&&
\\&&\left(\frac{Q_2^{(1)}(0)}{1!}x^1D^1+\frac{Q_1^{(0)}(0)}{0!}x^0D^0\right) D^1 \ +
\\&&
\\&&\left(\frac{Q_2^{(0)}(0)}{0!}x^0D^0\right) D^2
\\&&
\\&=& \underset{\mathcol{T_0}}{\underbrace{(a_2x^2D^2+a_1xD+a_0)}}+\underset{\mathcol{T_1}}{\underbrace{(b_1xD+b_0)}}D+\underset{\mathcol{T_2}}{\underbrace{(c_0)}}D^2.
\end{array}
\eeq
\end{example}

Theorem \ref{thm:classic} can be extended to arbitrary linear operators on $\R[x]$; reminiscent of a Laurent series from complex variables (see \cite[p. 222]{MH99}). 

\begin{theorem}\label{thm:classic2}
Let $T:\R[x]\to\R[x]$ be an arbitrary linear operator, 
\beq
T:=\sum_{k=0}^\infty Q_k(x)D^k. 
\eeq
Define the family of sequences, 
\beq
\{b_{n,k}\}_{k=0}^\infty:=\left\{\sum_{j=0}^k\binom{k}{j}Q_{j+n}^{(j)}(0)\right\}_{k=0}^\infty,\ \ \ \ \ \ n\in\Z, 
\eeq
where we take $Q_{j+n}^{(j)}(0)=0$ for $n+j<0$. For each $n\in\Z$, define the classical diagonal differential operator, 
\beq
T_n[x^k]:=b_{n,k}x^k. 
\eeq
Then, 
\beq\label{eq:diawriting2}
T=\sum_{n=1}^{\infty} T_{-n}D^{-n} + \sum_{n=0}^\infty T_nD^n, 
\eeq
where we define $D\cdot D^{-1}=1$. Furthermore, the representation in \eqref{eq:diawriting2} is unique. 
\end{theorem}

\begin{proof}
We first note that for each $n\in\N_0$, $ T_n=\sum_{k=0}^\infty \frac{Q_{k+n}^{(k)}(0)}{k!}x^kD^k$ (see Theorem \ref{thm:comb}). Similar to the proof of Theorem \ref{thm:classic}, each term from the $T_n$'s are in one-to-one correspondence with each term in the $Q_k$'s. Thus, a change of index yields, 
\beq
T=\sum_{n=0}^\infty Q_n(x)D^n = \sum_{n=0}^\infty \left(\sum_{k=0}^\infty \frac{Q_k^{(k)}(0)}{k!}x^k\right) D^n = \sum_{n=-\infty}^\infty \left(\sum_{k=0}^\infty \frac{Q_{k+n}^{(k)}(0)}{k!}x^k\right) D^{k+n} = \sum_{n=-\infty }^\infty T_n D^n. \nonumber\tag*{\qedhere}
\eeq
\end{proof}

\begin{example}
We provide another example demonstrating Theorem \ref{thm:classic2}. Define the differential operator, 
\beq
T:=(a_2x^2+b_1x+c_0)D^2+(z_1x^2+a_1x+b_0)D+(y_0x^2+z_0x+a_0), 
\eeq
where $y_0,z_1,z_0,a_2,a_1,a_0,b_1,b_0,c_0\in\R$. Using Theorem \ref{thm:classic2}, we rewrite $T$ in terms of $T_n$'s, 
\begin{align}
T=& \underset{T_{-2}}{\underbrace{(y_0x^2D^2)}}\mathcol{D^{-2}}\ +               
\\& \underset{T_{-1}}{\underbrace{(z_1x^2D^2+z_0xD)}}\mathcol{D^{-1}}\ +        
\\& \underset{T_0}{\underbrace{(a_2x^2D^2+a_1xD+a_0)}}\mathcol{D^0}\ +
\\& \underset{T_1}{\underbrace{(b_1xD+b_0)}}\mathcol{D^1}\ +
\\& \underset{T_2}{\underbrace{(c_0)}}\mathcol{D^2}\ . 
\end{align}
\end{example}

\begin{example}
It is possible for representation \eqref{eq:diawriting2} to be ``transcendental'' in both directions. Consider the differential operator, 
\beq
T:=\sum_{k=0}^\infty (x^{2k}+1)D^k. 
\eeq
Then for $n\in\N$, $T_{-n}=x^{2n}D^{2n}$ and for $n\in\N_0$, $T_{n}=1$. Hence, 
\begin{align}
T=& \cdots+T_{-2}D^{-2}+T_{-1}D^{-1}+T_0 D^0+T_1D^1+T_2D^2+\cdots \\
=&\cdots+(x^4D^4)D^{-2}+(x^2D^2)D^{-1}+(1)D^0+(1)D^1+(1)D^2+\cdots. 
\end{align} 
\end{example}

Upon attaining the representation \eqref{eq:diawriting} in Theorem \ref{thm:classic}, we direct our attention to the property of hyperbolicity preservation. If $T$ in equation \eqref{eq:diawriting}, is hyperbolicity preserving, then what properties do the $T_n$'s possess? One might hope that the $T_n$'s also enjoy the property of hyperbolicity preservation. This hope would certainly be warranted since, in fact, $T_0$ always possess the property of hyperbolicity preservation in a diagonal differential operator (see \cite{Bat14a} and \cite[Theorem 158, p. 145]{Pio07}). In addition, classical multiplier sequences and operators of the form $f(xD)$ and $f(D)$, from the Hermite-Poulain \cite[p. 4]{Obr63} and Laguerre Theorems \cite[Satz 3.2]{Obr63}, trivially have $T_n$'s that are hyperbolicity preserving. However, in general, our hope is false as the next several examples will demonstrate. The following Tur\'an type inequality, equation \eqref{eq:tti}, will be of great use. 

\begin{theorem}[{R. Bates and R. Yoshida \cite[(2013)]{BY13}}]\label{thm:BYform}
Let $a,b,c,r_1,r_2,r_3\in\R$. Define polynomials $Q_2(x)=a(x-r_1)(x-r_2)$, $Q_1(x)=b(x-r_3)$, and $Q_0(x)=c$. Then $T$ is hyperbolicity preserving, where 
\beq
T:=Q_2(x)D^2+Q_1(x)D+Q_0(x), 
\eeq 
if and only if $a, b, c$ are of the same sign and
\beq\label{eq:tti}
b^2\left(\frac{(r_1-r_3)(r_3-r_2)}{(r_1-r_2)^2}\right)-ac\ge 0. 
\eeq
We take $\left(\frac{(r_1-r_3)(r_3-r_2)}{(r_1-r_2)^2}\right)=\frac{1}{4}$ when $r_1=r_2=r_3$. If $r_1=r_2$ and $r_1\not=r_3$, then $T$ is not hyperbolicity preserving. 
\end{theorem}

\begin{remark}
For clarity, we point out that the condition that $a,b,c$ be of the same sign, in Theorem \ref{thm:BYform}, cannot be removed. For example, the following operator satisfies equation \eqref{eq:tti} but not the necessary sign condition of the leading coefficients, 
\beq
T:=(x-1)(x+1)D^2-2xD+1.
\eeq
Hence, $T$ is not hyperbolicity preserving, as can be seen since $T[x^2]=-x^2-2$. 
\end{remark}

\begin{example}\label{ex:qua}
Consider the following differential operator, 
\begin{align}
T:=&(x-2)(x+1)D^2+3(x+1/2)D+1 \\
=& (-2)D^2+(-xD+3/2)D+(x^2D^2+3xD+1) \\
=& T_2D^2+T_1D+T_0. 
\end{align}
By an application of Theorem \ref{thm:BYform}, operator $T$ is certainly hyperbolicity preserving, 
\beq
3^2\left(\frac{\lr{2-(-1/2)}\lr{(-1/2)-(-1)}}{\lr{(-1)-2}^2}\right)-1\cdot 1 = \frac{1}{4} \ge 0. 
\eeq
However, $T_1=-xD+3/2$ (see \eqref{eq:tn}) is not a hyperbolicity preserver, since $T_1[x^2-1]=(-1/2)x^2-3/2$. 
\end{example}

\begin{example}\label{ex:leg}
Consider the Legendre basis of polynomials, $\{P_n(x)\}_{n=0}^\infty$, that satisfy the differential equation (Definition \ref{def:poly}), 
\beq\label{eq:legendrehp}
((x^2-1)D^2+(2x)D+1)P_n(x) = (n^2+n+1)P_n(x). 
\eeq
Equation \eqref{eq:legendrehp} was first verified to be hyperbolicity preserving by K. Blakeman, E. Davis, T. Forg\'acs, and K. Urabe \cite[Lemma 5]{BDFU12}. We re-verify that $(x^2-1)D^2+(2x)D+1$ is a hyperbolicity preserver using the calculation in Theorem \ref{thm:BYform}, 
\beq
2^2\left(\frac{(1-0)(0-(-1))}{(-1-1)^2}\right)-1\cdot 1 = 1-1=0 \ge 0. 
\eeq
Hence, compositions are hyperbolicity preserving, and thus, $T$ is hyperbolicity preserving, where $T[P_n(x)]:=(n^2+n+1)^3P_n(x)$. We calculate the differential form of $T$ (see Theorem \ref{thm:piotr}),   
\begin{align}
T=&((x^2-1)D^2+(2x)D+1)^3 
\\=&(x^6-3x^4+\mathcol{3x^2}-1)D^6+
\\&(18x^5-36x^3+\mathcol{18x})D^5+
\\&(101x^4-130x^2+\mathcol{29})D^4+
\\&(208x^3-160x)D^3+
\\&(145x^2-57)D^2+
\\&(26x)D+
\\&1.
\end{align}
Consider the highlighted terms of from above to calculate $T_4$ (see \eqref{eq:tn}), 
\beq
T_4=3x^2D^2+18xD+29. 
\eeq
From Theorem \ref{thm:BYform} we infer that operator $T_4$ fails to be hyperbolicity preserving, \beq
18^2\left(\frac{1}{4}\right)-3\cdot 29 = 81 - 87 = -6 < 0. 
\eeq
\end{example}

\begin{example}\label{ex:sher}
Due to A. Piotrowski (see \cite[Lemma 157, p. 145]{Pio07}), affine transforms ($\{c_nB_n(\alpha x+\beta)\}_{n=0}^\infty$, $c_n,\alpha,\beta\in\R$, $c_n,\alpha\not=0$) share the same multiplier sequence class as the basis $\{B_n(x)\}_{n=0}^\infty$. Let us consider then an affine transform of the Hermite polynomials, $\{H_n(x\pm 3)\}_{n=0}^\infty$, and a multiplier sequence for these shifted Hermite polynomials, $\{n^2+n+1\}_{n=0}^\infty$ (see Theorem \ref{thm:msherm}). Thus $T$ is hyperbolicity preserving, where $T[H_n(x\pm 3)]=(n^2+n+1)H_n(x\pm 3)$. We calculate the differential form of $T$ (see Theorem \ref{thm:piotr}),
\beq\label{eq:shercalc}
T = \left(\frac{1}{4}\right)D^4+\lr{\mathcol{-x}\mp 3}D^3+\lr{x^2\pm 6x+\mathcol{\frac{15}{2}}}D^2+\lr{2x\pm 6}D+\lr{1}.
\eeq
From the highlighted items in \eqref{eq:shercalc} we formulate $T_2=-xD+15/2$ (see \eqref{eq:tn}) and note that $T_2$ is not hyperbolicity preserving since $T[2x^8-2x^6]=-x^8-3x^6$. 

It is intriguing to see that while affine transforms share multiplier sequence classes, the $T_n$'s in equation \eqref{eq:diawriting} may not share in the property of hyperbolicity preservation. Hence, as we will see in Theorem \ref{thm:hermiteissums} and \ref{thm:herm}, the Hermite polynomials are distinguished amongst all affine transforms of the Hermite polynomials. 
\end{example}

\begin{example}\label{ex:slag}
Consider the shifted Laguerre polynomials (see \cite[Lemma 157, p. 145]{Pio07}), $\{L_n(x+2)\}_{n=0}^\infty$, and a multiplier sequence for these shifted Laguerre polynomials, $\{n\}_{n=0}^\infty$ (see Theorem \ref{thm:mslag}). Thus $T$ is hyperbolicity preserving, where $T[L_n(x+2)]=nL_n(x+2)$ and 
\beq
T = (\mathcol{-x}-2)D^2+(x+\mathcol{1})D+(0).
\eeq
Consider the operator formed by the highlighted terms, $T_1=-xD+1$. Operator $T_1$ fails to preserve hyperbolicity since $T_1[x^2-1]=-x^2-1$. (See also Question \ref{que:que5} in the open problems.) 
\end{example}

\begin{example}\label{ex:malo}
A more technical example is the following. Using the generalized Malo-Schur-Szeg\"o Composition Theorem \cite{Bru49,CC04} it can be shown that, given $p(x)=(x+1)^3$, 
\begin{align}
T:&=-\frac{1}{6}p'''(x)D^3+\frac{1}{2}p''(x)D^2-p'(x)D+p(x) \\
 &=-D^3+(\mathcol{3x}+3)D^2+(-3x^2-6x\mathcol{-3})D+(x^3+3x^2+3x+1)
\end{align}
is hyperbolicity preserving \cite[p. 47]{Yos13}. Define $T_{1}:=3xD-3$ (see \eqref{eq:tn}) and note that $T_{1}[x^2-1]=3x^2+3$, thus $T_1$ is not hyperbolicity preserving. 
\end{example}

\begin{example}\label{ex:notdia}
Another example involving $Q_k$'s, where $\deg(Q_k(x))>k$ for some of the $k$'s. Using the Hermite-Poulain Theorem \cite[p. 4]{Obr63} it can be shown that the non-diagonalizable operator, 
\beq
T:=(\mathcol{x^2}+2x+1)D^2-(x^2+2x+\mathcol{1}), 
\eeq
preservers hyperbolicity. The operator $T_0=x^2D^2-1$ (see \eqref{eq:tn}) is not a hyperbolicity preserver, since $T_0[x^2-1]=x^2+1$. This example is even more interesting considering the fact that, in general, $W_0$ is always hyperbolicity preserving, whenever $W$ is any arbitrary diagonal differential hyperbolicity preserver (see \cite{Bat14a}). 
\end{example}

By now the reader has hopefully been convinced that Examples \ref{ex:qua}-\ref{ex:notdia} demonstrate the very high sensitivity of the following results; namely, for Hermite or Laguerre multiplier sequences the $T_n$'s in \eqref{eq:tn} from Theorem \ref{thm:classic} are hyperbolicity preservers. It is surprising, that not only will each $T_n$ be hyperbolicity preserving, the family of sequences, $\{b_{n,k}\}_{k=0}^\infty$ (see \eqref{eq:aform}), turn out to be more Hermite or Laguerre multiplier sequences, respectively. In this sense every Hermite or Laguerre multiplier sequence generates an entire family of additional Hermite or Laguerre multiplier sequences.

\vspace{.25in}
\section{Operator Diagonalizations of Hermite Multiplier Sequences}

Our main goal in this section is to demonstrate for hyperbolicity preserving Hermite diagonal differential operators, each $T_n$ defined in Theorem \ref{thm:classic} is hyperbolicity preserving. This will be done in two phases. First we will find a formula for $b_{n,k}$ (see \eqref{eq:aform}). Second, we will show that, for each $n\in\N_0$, $\{b_{n,k}\}_{k=0}^\infty$ a Hermite multiplier sequence and hence $\{b_{n,k}\}_{k=0}^\infty$ is also a classical multiplier sequence, i.e. each $T_n$ is hyperbolicity preserving. 

\begin{lemma}\label{lem:zeroherm}
For $k,j\in\N_0$, the $k^\text{th}$ derivative of the $(k+2j+1)^\text{th}$ and $(k+2j)^\text{th}$ Hermite polynomials \textup{(}see Definition \ref{def:poly}\textup{)} evaluated at zero is, 
\beq
H_{k+2j+1}^{(k)}(0)=0\ \ \ \text{and}\ \ \ H_{k+2j}^{(k)}(0)=\frac{(k+2j)!2^k(-1)^j}{j!}. 
\eeq
\end{lemma}

\begin{theorem}\label{thm:hermiteissums}
Let $T$ be a Hermite diagonal differential operator, $T[H_n(x)]:=\gamma_n H_n(x)$, where $\{\gamma_n\}_{n=0}^\infty$ a sequence of real numbers. Then there is a sequence of polynomials, $\{Q_k(x)\}_{k=0}^\infty$, and a sequence of classical diagonal differential operators, $\{T_n\}_{n=0}^\infty$, such that 
\beq\nonumber
T[H_n(x)]:=\left(\sum_{k=0}^\infty Q_k(x)D^k\right) H_n(x) = \left(\sum_{k=0}^\infty T_k D^k \right) H_n(x) = \gamma_n H_n(x). 
\eeq
Then, for each $n\in\N_0$, 
\beq\nonumber
\{b_{2n+1,m}\}_{m=0}^\infty = \{0\}_{m=0}^\infty
\eeq
and
\beq\nonumber
\{b_{2n,m}\}_{m=0}^\infty:=\left\{\sum_{k=0}^m \binom{m}{k} \frac{(-1)^n}{n!2^n} \left(\sum_{j=0}^n \binom{n}{j} \frac{g_{k+n+j}^*(-1)}{2^{j}}\right)\right\}_{m=0}^\infty, 
\eeq
where $T_n[x^m]=b_{n,m}x^m$ for every $n,m\in\N_0$. 
\end{theorem}

\begin{proof}
The exists of the sequences $\{Q_k(x)\}_{k=0}^\infty$ and $\{T_k\}_{k=0}^\infty$ are established by Theorem \ref{thm:piotr} and \ref{thm:classic}. We now begin with the remarkable representation formula of T. Forg\'acs and A. Piotrowski that computes the $Q_k$'s in any Hermite diagonal differential operator \cite[Theorem 3.1]{FP13b}, 
\beq\label{eq:forgacs}
Q_k(x)=\sum_{j=0}^{[k/2]} \frac{(-1)^j}{j!(k-2j)!2^{k-j}}g_{k-j}^*(-1)H_{k-2j}(x). 
\eeq
This formula yields the following expressions for all $k,n\in\N_0$, 
\beq\label{eq:calc1}
Q_{k+2n+1}^{(k)}(0)=0, 
\eeq
and
\beq\label{eq:calc2}
Q_{k+2n}^{(k)}(0)=\frac{(-1)^n}{n!2^n}\sum_{j=0}^{n} \binom{n}{j}\frac{g_{k+n+j}^*(-1)}{2^{j}}. 
\eeq

Equations \eqref{eq:calc1} and \eqref{eq:calc2} could have been calculated using the recursive formula of Theorem \ref{thm:piotr}, if one knew, \textit{a priori}, the importance of the $g_{k-j}^*(-1)$'s in formula \eqref{eq:forgacs}. However, this dependence was not made apparent until formula \eqref{eq:forgacs} was uncovered. 

Let us now verify \eqref{eq:calc1} and \eqref{eq:calc2}. Equation \eqref{eq:calc1} is obvious from formula \eqref{eq:forgacs} and the fact that the Hermite polynomials alternate between even and odd polynomials. We now establish \eqref{eq:calc2} using formula \eqref{eq:forgacs} and Lemma \ref{lem:zeroherm} as follows: 
\begin{align}
Q_{k+2n}^{(k)}(0) 
&=\sum_{j=0}^{[(k+2n)/2]} \frac{(-1)^j}{j!(k+2n-2j)!2^{k+2n-j}}g_{k+2n-j}^*(-1)H_{k+2n-2j}^{(k)}(0) \\
&=\sum_{j=0}^{n} \frac{(-1)^j}{j!(k+2(n-j))!2^{k+n+(n-j)}}g_{k+n+(n-j)}^*(-1)H_{k+2(n-j)}^{(k)}(0) \\
&=\sum_{j=0}^{n} \frac{(-1)^{n-j}}{(n-j)!(k+2j)!2^{k+n+j}}g_{k+n+j}^*(-1)H_{k+2j}^{(k)}(0) \\
&=\sum_{j=0}^{n} \frac{(-1)^{n-j}}{(n-j)!(k+2j)!2^{k+n+j}}g_{k+n+j}^*(-1)\left(\frac{(k+2j)!2^k(-1)^j}{j!}\right) \\
&=\frac{(-1)^n}{n!2^n}\sum_{j=0}^{n} \binom{n}{j}\frac{g_{k+n+j}^*(-1)}{2^{j}}. 
\end{align}
We finish the proof by using formula \eqref{eq:aform}. 
\end{proof}

With the aid of what has been shown thus far, we are now in a position to demonstrate our main result, that every Hermite multiplier sequence is the unique sum of classical multiplier sequences. That is, for Hermite multiplier sequences, each $T_n$ in equation \eqref{eq:tn} is hyperbolicity preserving. The spirit of the following argument will be the establishment of a Rodrigues type formula that relates each governing entire function, $\sum_{k=0}^\infty \frac{b_{n,k}}{k!}x^k$, of each $T_n$, with the entire function that defines the hyperbolicity properties of $T$ itself, $\sum_{k=0}^\infty \frac{\gamma_k}{k!}x^k$ (see Theorem \ref{thm:ms} and \ref{thm:msherm}). 

\begin{theorem}\label{thm:herm}
Let $\{\gamma_k\}_{k=0}^\infty$ is a non-trivial Hermite multiplier sequence and let $\{g_k^*(x)\}_{k=0}^\infty$ be the reversed Jensen polynomials associated with $\{\gamma_k\}_{k=0}^\infty$. Then, for each $n\in\N_0$,
\beq\label{eq:crazyherm}
\{b_{n,m}\}_{m=0}^\infty:=\left\{\sum_{k=0}^m \binom{m}{k} \frac{(-1)^n}{n!2^n} \left(\sum_{j=0}^n \binom{n}{j} \frac{g_{k+n+j}^*(-1)}{2^{j}}\right)\right\}_{m=0}^\infty, 
\eeq
is a Hermite multiplier sequence. 
\end{theorem}

\begin{proof}
By assumption, $\{\gamma_k\}_{k=0}^\infty$ is a Hermite multiplier sequence. Hence, by Theorem \ref{thm:msherm}, if 
\beq
f(x):=\sum_{k=0}^\infty \frac{f^{(k)}(0)}{k!} x^k := \sum_{k=0}^\infty \frac{g_k^*(-1)}{k!} x^k, 
\eeq
then, either $f(x)\in\lp^s$ or $e^{2x}f(x)\in\lp^a$. We wish to show that, $\{b_{n,m}\}_{m=0}^\infty$, is a Hermite multiplier sequence; thus using Theorem \ref{thm:msherm} we must show that if
\beq\label{h_n's}
h_n(x):=\sum_{m=0}^\infty \left(\sum_{k=0}^m \binom{m}{k} b_{n,k} (-1)^{m-k} \right) \frac{x^m}{m!}, 
\eeq
then either $h_n(x)\in\lp^s$ or $e^{2x}h_n(x)\in\lp^a$. We use Theorem \ref{thm:comb} and perform the following calculation, 
\begin{align}
h_n(x)
&=\sum_{m=0}^\infty \left(\sum_{k=0}^m \binom{m}{k} b_{n,k} (-1)^{m-k} \right) \frac{x^m}{m!} \\
&=\sum_{k=0}^\infty \left(\frac{(-1)^n}{n!2^n} \left(\sum_{j=0}^n \binom{n}{j} \frac{g_{k+n+j}^*(-1)}{2^{j}}\right) \right)\frac{x^k}{k!} \\
&=\frac{(-1)^n}{n!2^n} \sum_{j=0}^n \binom{n}{j}\frac{1}{2^{j}} \sum_{k=0}^\infty \left(\frac{g_{k+n+j}^*(-1)}{k!}\right) x^k \\
&=\frac{(-1)^n}{n!2^n} \sum_{j=0}^n \binom{n}{j} \frac{1}{2^j} D^{n+j} f(x) \\
&=\frac{(-1)^n}{n!4^n} D^n \left(\sum_{j=0}^n \binom{n}{j} D^{j} 2^{n-j}\right) f(x) \\
&=\frac{(-1)^n}{n!4^n} D^n (2+D)^n f(x) \\
&=\frac{(-1)^n}{n!4^n} D^n e^{-2x} D^n e^{2x} f(x). \label{eq:hermrecur}
\end{align}
Hence, if $f(x)\in\lp^s$, then $h_n(x)\in\lp^s$ and if $e^{2x}f(x)\in\lp^a$, then $e^{2x}h_n(x)\in\lp^a$ (see also Remark \ref{rm:lpinc}). 
\end{proof}

Equation \eqref{eq:hermrecur} yields a little more information than Theorem \ref{thm:herm}, in particular we derive the recursive formula, 
\beq
h_n(x)=\frac{-1}{4n}De^{-2x}De^{2x} h_{n-1}(x),\ \ \ \ \ (n\ge 1,\ h_0(x):=f(x)).  
\eeq
Hence, only $T_n$ needs to be diagonalizable with a Hermite multiplier sequence to establish that $T_{n+1}$ is also diagonalizable with a Hermite multiplier sequence. 

Given a Hermite diagonal differential operator, $T[H_n(x)]=\gamma_n H_n(x)$, $\gamma_n\in\R$, (see Definition \ref{def:diagonalize}), then $T_0$ (see \ref{eq:tn}) diagonalizes with the same eigenvalue sequence, namely $T_0[x^n]=\gamma_n x^n$. In fact, this is more generally known (see \cite{Bat14a}). This indicates that if one assumes each operator $T_n$ yields a Hermite multiplier sequence, then Theorem \ref{thm:herm} has a trivial converse, in the sense that if one assumes each $T_n$ diagonalizes with a Hermite multiplier sequence then $T$ itself will also be hyperbolicity preserving. However, what if one only assumes that each $T_n$ is hyperbolicity preserving? Must $T$ be hyperbolicity preserving? We answer this question in the negative, with the following examples. 

\begin{example}\label{ex:notherm1}
Consider the following Hermite diagonal operator that is not hyperbolicity preserving (see Theorem \ref{thm:msherm}), 
\beq
T[H_n(x)]:=\left((-1)^{n+1}(n-1)\right)H_n(x).
\eeq
Thus we calculate, 
\beq
w(x):=\sum_{k=0}^\infty \frac{(-1)^{k+1}(k-1)}{k!}x^k = (x+1)e^{-x}. 
\eeq
Hence, using equation \eqref{eq:hermrecur} (note, $f(x)=e^{-x}w(x)$ (see Theorem \ref{thm:msherm})), we can calculate the $h_n$'s, 
\begin{align}
h_0(x)& := \sum_{k=0}^\infty \frac{Q_k^{(k)}(0)}{k!}x^k = (x+1)e^{-2x},\\ 
h_1(x)& := \sum_{k=0}^\infty \frac{Q_{k+2}^{(k)}(0)}{k!}x^k = \frac{1}{2}e^{-2x},\ \ \ \ \ \text{and}\\ 
h_n(x)& := \sum_{k=0}^\infty \frac{Q_{k+2n}^{(k)}(0)}{k!}x^k = 0,\ \ \ \ \ \text{for}\ n\ge 2. 
\end{align}
Hence, 
\begin{align}
T&=1-xD+\sum_{k=0}^\infty \Bigg(\frac{\overset{\underbrace{h_0^{(k+2)}(0)}}{\mathcol{k(-2)^{k+1}}}}{(k+2)!}x^{k+2}+\frac{\overset{\underbrace{h_1^{(k)}(0)}}{\mathcol{-(-2)^{k-1}}}}{k!}x^k\Bigg) D^{k+2} \\
&=T_0+T_2D^2.
\end{align}
Thus, 
\begin{align}
T_0[x^n]
&=\Bigg(1-xD+\sum_{k=0}^\infty\frac{k(-2)^{k+1}}{(k+2)!}x^{k+2}D^{k+2}\Bigg) x^n = \left((-1)^{n+1}(n-1)\right) x^n, \\
T_2[x^n]&=\left(\sum_{k=0}^\infty \left(\frac{-(-2)^{k-1}}{k!}x^k\right)D^{k}\right) x^n = \left(\frac{1}{2}(-1)^n\right) x^n,\ \ \ \ \ \text{and} \\
T_{2m}[x^n] &= \left(0\right)x^n=(0)x^n,\ \ \ \ \ \text{for}\ m\ge 2. 
\end{align}
We see that for every $n\ge 1$, $h_n(x)\in\lp^a$, hence $T_{2n}$ is hyperbolicity preserving (see Theorem \ref{thm:ms}). However, the original operator $T$ itself is not hyperbolicity preserving, as the following calculation shows,
\begin{align}
T[4x^2+2x-5]=&\ T[\overset{\underbrace{-3H_0(x)}}{(-3)}+\overset{\underbrace{H_1(x)}}{(2x)}+\overset{\underbrace{H_2(x)}}{(4x^2-2)}]\\
=&\ 1(-3)+0(2x)+(-1)(4x^2-2)=-4x^2-1.  
\end{align}
\end{example}

\begin{example}\label{ex:notherm2}
Consider another Hermite diagonal operator that does not preserve hyperbolicity (see Theorem \ref{thm:msherm}), $\{\gamma_k\}_{k=0}^\infty=\{(1/2)^k\}_{k=0}^\infty$; that is,  
\beq
T[H_n(x)]=\gamma_nH_n(x):=(1/2)^n H_n(x). 
\eeq
Using Theorem \ref{thm:classic} we write $T=\sum_{n=0}^\infty T_nD^n$, where $T_n[x^m]=b_{n,m}x^m$. We rewrite formula \eqref{eq:hermrecur} in terms of $b_{n,m}$'s and $\gamma_n$'s (see Theorem \ref{thm:comb}), 
\beq
\sum_{k=0}^\infty \frac{b_{n,k}}{k!}x^k = \frac{(-1)^n}{n!4^n} e^x D^n e^{-2x} D^n e^x \sum_{k=0}^\infty \frac{\gamma_k}{k!}x^k. 
\eeq
Since $\sum_{k=0}^\infty \frac{\gamma_k}{k!}x^k = e^{x/2}$, then 
\beq
\sum_{k=0}^\infty \frac{b_{n,k}}{k!}x^k = \frac{(-1)^n}{n!4^n}\left(-\frac{1}{2}\right)^n\left(\frac{3}{2}\right)^n e^{x/2}. 
\eeq
Thus $\sum_{k=0}^\infty \frac{b_{n,k}}{k!}x^k\in\lp^s$ for every $n\in\N_0$. Hence, $T_{2n}$ is hyperbolicity preserving for every $n\in\N_0$ (see Theorem \ref{thm:ms}), however, as noted above, $T$ is not hyperbolicity preserving (see Theorem \ref{thm:msherm}). 
\end{example}

\begin{example}
To demonstrate the usefulness of Theorem \ref{thm:herm}, consider the following example. How would one show that 
\beq
\{a_m\}_{m=0}^\infty:=\{m^{5/2}\}_{m=0}^\infty
\eeq
is not a multiplier sequence? Sequence $\{a_m\}_{m=0}^\infty$ satisfies the Tur\'an inequalities and is a positive, increasing sequence. Thus some well known methods do not work (see for example see \cite[p. 341]{Lev80}, concerning the Tur\'an inequalities). One could apply the sequence to $(1+x)^5$ to calculate to the fifth associated Jensen polynomial, 
\beq
= (5)x+(56.56\ldots)x^2+(155.88\ldots)x^3+(160)x^4+(55.90\ldots)x^5  
\eeq
and verify that this polynomial has non-real zeros, however this can prove to be quite tedious. Instead, we apply Theorem \ref{thm:herm} and calculate as summarized in Figure \ref{eq:bnthing}. 
\begin{figure}[h!]
\beq\nonumber
\barr{cclll}
&&\\
b_{0,n}&=&0,&1,&\cdots\\&&\\
b_{1,n}&=&-1.41,&-3.65,&\cdots\\&&\\
b_{2,n}&=&0.646,&0.804,&\cdots\\&&\\
b_{3,n}&=&\mathcol{-0.0238},&\mathcol{-0.020},&\cdots\\&&\\
\vdots&&\ \ \vdots&\ \ \vdots&\ddots
\earr
\eeq
\caption{Table of Hermite diagonal differential operator eigenvalues. \label{eq:bnthing}}
\end{figure}
Hence, after a few simple \textit{numerical} calculations we arrive at the highlighted portions in Figure \ref{eq:bnthing} and note that they are negative and increasing, so $\{b_{3,n}\}_{n=0}^\infty$ is not a Hermite multiplier sequence (see Theorem \ref{thm:msherm}). Thus, the original sequence, $\{a_m\}_{m=0}^\infty$, is not a Hermite multiplier sequence. Consequently, since $\{a_m\}_{m=0}^\infty$ is an increasing sequence that is not a Hermite multiplier sequence, by Theorem \ref{thm:msherm}, we conclude that $\{a_m\}_{m=0}^\infty$ cannot be a classical multiplier sequence. 
\end{example}

Our next task is to present several relationships between the polynomial coefficients, the $Q_k$'s, and the eigenvalues, the $\gamma_k$'s, in a Hermite diagonal differential operator, 
\beq
T[H_n(x)]:=\left(\sum_{k=0}^\infty Q_k(x) D^k\right)H_n(x)=\gamma_n H_n(x). 
\eeq
In general, in a diagonal differential operator, the relationship between the $Q_k$'s and the $\gamma_k$'s is not well understood, particularly in the context of hyperbolicity preservation. In special cases direct formulas have been found (see for example \eqref{eq:forgacs}) (cf. Theorem \ref{thm:piotr} and \cite[Proposition 216, p. 107]{Cha11}), but a general relation has not been derived that indicates the properties of the $Q_k$'s and the $\gamma_k$'s for arbitrary hyperbolicity preserving operators. Thus, whenever possible, it is beneficial to present formulas that highlight the nature of the $Q_k$'s in terms of the eigenvalues, the $\gamma_k$'s. Using calculation \eqref{eq:calc1} and \eqref{eq:calc2}, in Theorem \ref{thm:hermqk} we can provide another formula for the $Q_k$'s in a Hermite diagonal differential operator. 

\begin{theorem}\label{thm:hermqk}
Let $\{\gamma_n\}_{n=0}^\infty$ be a sequence of real numbers and $\{Q_k(x)\}_{k=0}^\infty$ be a sequence of real polynomials, such that 
\beq
T[H_n(x)]:=\left(\sum_{k=0}^\infty Q_k(x)D^k\right)H_n(x) = \gamma_n H_n(x),\ \ \ n\in\N_0. 
\eeq
Then for each $m\in\N_0$, 
\beq
Q_m(x)=\sum_{k=0}^{[m/2]} \frac{(-1)^k}{k!2^k} \left(\sum_{j=0}^k \binom{k}{j} \frac{g_{m-k+j}^*(-1)}{2^{j}}\right)\frac{x^{m-2k}}{(m-2k)!}, 
\eeq
where $\{g_k^*(x)\}_{k=0}^\infty$ are the associated reversed Jensen polynomials of $\{\gamma_n\}_{n=0}^\infty$. 
\end{theorem}

We also derive a complex formulation for the $Q_k$'s in a Hermite diagonal differential operator (Theorem \ref{thm:hermcomplex}). A heuristic argument of the proof of Theorem \ref{thm:hermcomplex} follows easily by considering the generating function of the Hermite polynomials (see \cite[p. 187]{Rai60}), 
\beq\label{eq:hermgen}
e^{2xt-t^2}=\sum_{n=0}^\infty \frac{H_n(x)}{n!} t^n. 
\eeq
We now calculate $T[e^{2xt-t^2}]$ in two ways, 
\begin{align} 
T[e^{2xt-t^2}]&=\left(\sum_{k=0}^\infty Q_k(x)D^k\right) e^{2xt-t^2}=e^{2xt-t^2}\sum_{k=0}^\infty Q_k(x)(2t)^k,\ \ \ \ \ \text{and} \\ 
T[e^{2xt-t^2}]&=T\left[\sum_{n=0}^\infty \frac{H_n(x)}{n!}t^n\right]=\sum_{n=0}^\infty \frac{\gamma_n H_n(x)}{n!}t^n. 
\end{align}
Hence, 
\beq\label{eq19}
\sum_{k=0}^\infty Q_k(x) (2t)^k = e^{-2xt+t^2} \left(\sum_{n=0}^\infty \frac{\gamma_n H_n(x)}{n!}t^n \right). 
\eeq
Thus, performing a Cauchy product on the right hand side of \eqref{eq19} and comparing the coefficients of $t^n$ on the right and left of \eqref{eq19}, for each $n\in\N_0$, we have, 
\begin{align}
Q_n(x)2^n 
&= \frac{1}{n!} \sum_{k=0}^n \binom{n}{k} \left.\left(\frac{d^{n-k}}{dt^{n-k}}e^{-2xt+t^2}\right)\right|_{t=0} \left.\left(\frac{d^{k}}{dt^{k}} \sum_{j=0}^\infty \frac{\gamma_j H_j(x)}{j!}t^j\right)\right|_{t=0} \\ 
&= \frac{1}{n!} \sum_{k=0}^n \binom{n}{k} \left.\frac{d^{n-k}}{dt^{n-k}} \sum_{j=0}^\infty \frac{H_j(ix)}{j!}(it)^j \right|_{t=0} \left.\frac{d^{k}}{dt^{k}} \sum_{j=0}^\infty \frac{\gamma_j H_j(x)}{j!}t^j\right|_{t=0} \\ 
&= \frac{1}{n!} \sum_{k=0}^n \binom{n}{k} i^{n-k}H_{n-k}(ix) \gamma_k H_k(x). 
\end{align}

\begin{remark}\label{rmk:careful}
We must be cautious with the argument above since $T[e^{2xt-t^2}]$ need not converge and hence is only calculated formally. However, even under formal assumptions there is no reason to assume that a differential representation of a linear operator will calculate the same formal series as the operator itself. That is, the calculation, 
\beq\label{eq:thingy}
T[e^{2xt-t^2}]=e^{2xt-t^2}\sum_{k=0}^\infty Q_k(x)(2t)^k, 
\eeq
has not been rigorously established.  
\end{remark}

\begin{theorem}\label{thm:hermcomplex}
Let $\{\gamma_n\}_{n=0}^\infty$ be a sequence of real numbers and $\{Q_k(x)\}_{k=0}^\infty$ be a sequence of real polynomials, such that 
\beq
T[H_n(x)]:=\left(\sum_{k=0}^\infty Q_k(x)D^k\right)H_n(x) = \gamma_n H_n(x),\ \ \ n\in\N_0. 
\eeq
Then for each $n\in\N_0$, 
\beq\label{eq:newherm}
Q_n(x)=\frac{1}{n!2^n}\sum_{k=0}^{n} \binom{n}{k} \gamma_k i^{n-k}H_{n-k}(ix)H_k(x).  
\eeq
\end{theorem}

\begin{proof}
Define 
\beq
\tilde{T}:=\sum_{k=0}^\infty Q_k(x)D^k, 
\eeq
where we define $Q_k(x)$ from equation \eqref{eq:newherm}. In the spirit of T. Forg\'acs and A. Piotrowski \cite[Theorem 3.1]{FP13b}, we need only to show that $\tilde{T}[H_n(x)]=\gamma_n H_n(x)$ for each $n\in\N_0$. We note that for $n,m\in\N_0$, $D^m H_n(x) = 2^m \binom{m}{n} n! H_{n-m}(x)$ \cite[p. 188]{Rai60}. We also note that $\binom{n}{k}\binom{k}{j} = \binom{n}{j}\binom{n-j}{k-j}$ (see \cite[p. 3]{Rio79}). Using the generating function of the Hermite polynomials, equation \eqref{eq:hermgen}, we now calculate 
\begin{align}
\tilde{T}[H_n(x)] 
&=\sum_{k=0}^n\left(\frac{1}{k!2^k}\sum_{j=0}^{k} \binom{k}{j} \gamma_j i^{k-j}H_{k-j}(ix)H_j(x)\right)\mathcol{D^k H_n(x)} \\
&=\sum_{k=0}^n\left(\frac{1}{k!2^k}\sum_{j=0}^{k} \binom{k}{j} \gamma_j i^{k-j}H_{k-j}(ix)H_j(x)\right)\mathcol{2^k\binom{n}{k}k! H_{n-k}(x)} \\
&=\sum_{j=0}^n \gamma_j H_j(x) \sum_{k=j}^{n} \binom{n}{k} \binom{k}{j} i^{k-j}H_{k-j}(ix)H_{n-k}(x) \\
&=\sum_{j=0}^n \gamma_j H_j(x) \sum_{k=0}^{n-j} \mathcol{\binom{n}{k+j} \binom{k+j}{j}} i^{k}H_{k}(ix)H_{(n-j)-k}(x) \\
&=\sum_{j=0}^n \gamma_j H_j(x) \sum_{k=0}^{n-j} \mathcol{\binom{n}{j} \binom{n-j}{k}} i^{k}H_{k}(ix)H_{(n-j)-k}(x) \\
&=\sum_{j=0}^n \binom{n}{j} \gamma_j H_j(x) \cdot\sum_{k=0}^{n-j} \binom{n-j}{k} \left(\left.\frac{d^{k}}{dt^k}e^{-2xt+t^2}\right|_{t=0}\right)\cdot \left(\left.\frac{d^{(n-j)-k}}{dt^{(n-j)-k}}e^{2xt-t^2}\right|_{t=0}\right) \\
&=\sum_{j=0}^n \binom{n}{j} \gamma_j H_j(x) \left.\frac{d^{n-j}}{dt^{n-j}}e^{-2xt+t^2}e^{2xt-t^2}\right|_{t=0} \\
&=\gamma_n H_n(x). \nonumber\tag*{\qedhere}
\end{align}
\end{proof}

We can also establish an interesting relationship between alternating Hermite diagonal differential operators and non-alternating Hermite diagonal differential operators. This will allow us to provide an alternate proof and a non-obvious extension of T. Forg\'acs and A. Piotrowski \cite[Theorem 3.7]{FP13b} (cf.~Example~\ref{ex:hermfinin}). 

\begin{theorem}\label{thm:herminter}
Let $\{\gamma_k\}_{k=0}^\infty$ be a sequence of real numbers. Define the Hermite diagonal differential operators, 
\beq
T[H_n(x)]:=\left(\sum_{k=0}^\infty Q_k(x)D^k\right)H_n(x) = \gamma_n H_n(x)
\eeq
and
\beq
\tilde{T}[H_n(x)]:=\left(\sum_{k=0}^\infty \tilde{Q}_k(x)D^k\right)H_n(x) = (-1)^n\gamma_n H_n(x). 
\eeq
Then for each $n\in\N_0$,  
\beq\label{eq:thingy23}
Q_n(x)=\frac{(-2)^n}{n!}\left(\sum_{k=0}^\infty \frac{\tilde{Q}_k(x)}{2^k}D^k\right) x^n
\eeq
and
\beq
\tilde{Q}_n(x)=\frac{(-2)^n}{n!}\left(\sum_{k=0}^\infty \frac{Q_k(x)}{2^k}D^k\right) x^n. 
\eeq
\end{theorem}

\begin{proof}
In light of Remark \ref{rmk:careful} and Theorem \ref{thm:hermcomplex}, we may conclude that, 
\beq\label{eq:thing8}
\sum_{k=0}^\infty Q_k(x)(2t)^k = e^{-2xt+t^2}\left(\sum_{n=0}^\infty \frac{\gamma_n H_n(x)}{n!} t^n\right)
\eeq
and
\beq
\sum_{k=0}^\infty \tilde{Q}_k(x)(2t)^k = e^{-2xt+t^2}\left(\sum_{n=0}^\infty \frac{(-1)^n\gamma_n H_n(x)}{n!} t^n\right);
\eeq
i.e., as formal power series in $t$, the coefficients are equal (see \cite{Niv69} or \cite[p. 130]{Rot02}). Hence, after substitution of $t\to -t$, we have
\begin{align}
e^{-4xt}\sum_{k=0}^\infty Q_k(x)(-2t)^k 
&= e^{-4xt}\left(e^{2xt+t^2} \left(\sum_{n=0}^\infty \frac{(-1)^n\gamma_n H_n(x)}{n!} t^n\right) \right)\\
&= e^{-2xt+t^2} \sum_{n=0}^\infty \frac{(-1)^n\gamma_n H_n(x)}{n!} t^n \\
&= \sum_{k=0}^\infty \tilde{Q}_k(x)(2t)^k. 
\end{align}
Thus, 
\begin{align}
\tilde{Q}_n(x) 
&= \frac{1}{n!2^n}\frac{d^n}{dt^n} \left. e^{-4xt}\sum_{k=0}^\infty Q_k(x)(-2t)^k \right|_{t=0} \\
&= \frac{1}{n!2^n}\sum_{k=0}^n \binom{n}{k} (-4x)^{n-k} (-2)^k k! Q_k(x) \\
&= \frac{(-2)^n}{n!}\sum_{k=0}^n  \binom{n}{k} \frac{x^{n-k}}{2^k} k! Q_k(x) \\
&= \frac{(-2)^n}{n!}\left(\sum_{k=0}^n  \frac{Q_k(x)}{2^k}D^k \right)x^n. 
\end{align}
By symmetry, equation \eqref{eq:thingy23} also holds. 
\end{proof}

\begin{theorem}[{\hspace{1sp}\cite[Theorem 3.7]{FP13b}}]\label{thm:hermrealzeros}
Let $\{\gamma_k\}_{k=0}^\infty$ be a non-trivial Hermite multiplier sequence,  
\beq
T[H_n(x)]:=\left(\sum_{k=0}^\infty Q_k(x)D^k\right)H_n(x) = \gamma_n H_n(x). 
\eeq 
Then each $Q_k(x)$ has only real zeros. 
\end{theorem}

\begin{proof}
If $\{\gamma_k\}_{k=0}^\infty$ is a Hermite multiplier sequence, then $\{(-1)^k\gamma_k\}_{k=0}^\infty$ is also a Hermite multiplier sequence \cite[Proposition 119, p. 98]{Pio07}. Hence, 
\beq
\sum_{k=0}^\infty \tilde{Q}_k(x)D^k, 
\eeq
is a hyperbolicity preserver. Thus, using the Borcea-Br\"and\'en Theorem \cite[Theorem 5]{BB09} (which requires non-trivial), we conclude that the operator, 
\beq
\sum_{k=0}^\infty \frac{\tilde{Q}_k(x)}{2^k}D^k, 
\eeq
is also a hyperbolicity preserver. In particular, by Theorem \ref{thm:herminter}, for each $n\in\N_0$,
\beq
Q_n(x)=\frac{(-2)^n}{n!}\left(\sum_{k=0}^\infty \frac{\tilde{Q}_k(x)}{2^k}D^k\right) x^n, 
\eeq
has only real zeros. 
\end{proof}

Theorem \ref{thm:herminter} actually shows that for every $k\in\N_0$, $Q_k(x)$ and $Q_{k+1}(x)$ have real interlacing zeros \cite[Remark 6, p. 5]{Cha11}; i.e., for every $\alpha,\beta\in\R$, $k\in\N_0$, $\alpha Q_k(x)+\beta Q_{k+1}(x)$ has only real zeros. 

We also note that Theorem \ref{thm:herminter} seems to indicate that only the polynomials $\{x^n\}_{n=0}^\infty$ are needed to establish that a Hermite diagonal differential operator is a hyperbolicity preserver. This observation provides us a new algebraic characterization of Hermite multiplier sequences (cf. \cite[Theorem 46, p. 44]{Pio07}). 

\begin{theorem}\label{thm:minimal}
Let $\{\gamma_n\}_{n=0}^\infty$ be a non-zero, positive, classical multiplier sequence of real numbers and let $T$ be a Hermite diagonal differential operator, where $T[H_n(x)]:=\gamma_n H_n(x)$ for every $n\in\N_0$. Then $T$ is hyperbolicity preserving if and only if, 
\beq
T[x^n]\in\lp, 
\eeq
for every $n\in\N_0$. 
\end{theorem}

\begin{proof}
In order to establish the non-trivial direction, it suffices to show $T$ is hyperbolicity preserving; i.e., $\{\gamma_n\}_{n=0}^\infty$ is a Hermite multiplier sequence. We will make use of the fact that $H'_{n}(x)=2nH_{n-1}(x)$ for every $n\in\N$ \cite[p. 188]{Rai60}. By assumption, for each $n\ge 2$, the following polynomial has only real zeros (see \cite[p. 194]{Rai60} for the Hermite expansion of $x^n$), 
\begin{align}
D^{n-2}T[x^n]
&=D^{n-2}T\left[\frac{n!}{2^n}\sum_{k=0}^{[n/2]} \frac{1}{k!(n-2k)!}H_{n-2k}(x)\right] \\
&=D^{n-2}\frac{n!}{2^n} \sum_{k=0}^{[n/2]} \frac{\gamma_{n-2k}}{k!(n-2k)!} H_{n-2k}(x) \\
&=\frac{n!}{2^n}\left(\gamma_n \frac{2^{n-2}}{2!}H_2(x)+\gamma_{n-2} \frac{2^{n-2}}{1!} H_0(x)\right) \\
&=n!\left(\frac{\gamma_n}{8}(4x^2-2)+\frac{\gamma_{n-2}}{4}(1)\right) \\
&=\frac{n!\gamma_n}{4}\left(2x^2+\left(\frac{\gamma_{n-2}}{\gamma_n}-1\right)\right).
\end{align}
Hence, $\frac{\gamma_{n-2}}{\gamma_{n}}\le 1$ for every $n\ge 2$. Following the outline of A. Piotrowski \cite[Theorem 127, p. 107]{Pio07},  since $\{\gamma_n\}_{n=0}^\infty$ is assumed to be a multiplier sequence, then the Tur\'an inequalities hold, $\gamma_{n-1}^2-\gamma_{n-2}\gamma_n\ge 0$ for every $n\ge 2$. Hence, for each $n\ge 2$,
\beq
1\le \frac{\gamma_{n}}{\gamma_{n-2}}\le \left(\frac{\gamma_{n-1}}{\gamma_{n-2}}\right)^2. 
\eeq
Thus, $\gamma_{n-2}\le \gamma_{n-1}$ for $n\ge 2$, and therefore $\{\gamma_n\}_{n=0}^\infty$ is a Hermite multiplier sequence (see Theorem \ref{thm:lpinc} and \ref{thm:msherm}).  
\end{proof}

\vspace{.25in}
\section{Operator Diagonalizations of Laguerre Multiplier Sequences}

The main objective of this section is exactly the same as that of the previous. We provide a few preliminary remarks for Laguerre multiplier sequences, we then find a formula for the $b_{n,k}$'s (see \eqref{eq:aform}), and finally we show that the $b_{n,k}$'s (see \eqref{eq:aform}) that arise from a Laguerre multiplier sequence yield more Laguerre multiplier sequences. The subtlety of the proceeding results can be seen in Examples \ref{ex:qua}-\ref{ex:notdia}, particularly Example \ref{ex:slag}. 

\begin{lemma}\label{lem:zerolag}
For $k,n\in\N_0$, the $k^\text{th}$ derivative of the $n^\text{th}$ Laguerre polynomial \textup{(}Definition \ref{def:poly}\textup{)} evaluated at zero is, 
\beq
L_{n}^{(k)}(0)=\binom{n}{k}(-1)^k. 
\eeq
\end{lemma}

\begin{lemma}\label{lm:horrible}
Let $n$, $m$, and $p$ be integers. We then have the following combinatorial identity, 
\begin{align}
\sum_{k=0}^n\sum_{j=0}^m\binom{m}{j}&\binom{k-j}{p-j}\binom{p}{k-j}\binom{n+1}{k+m-j} \nonumber\\
& = \binom{n+1}{p}\binom{n+1}{m} - \binom{n+1-m}{p-m}\binom{p}{n+1-m}. 
\end{align}
\end{lemma}

\begin{proof}
We first note that $\binom{n+1-m}{p-m}\binom{p}{n+1-m}$ can be added to the summation, hence, we wish to show, 
\beq\label{eq:eq13}
\sum_{k=0}^{n+1}\sum_{j=0}^{m}\binom{m}{j}\binom{k-j}{p-j}\binom{p}{k-j}\binom{n+1}{k+m-j}=\binom{n+1}{p}\binom{n+1}{m}.
\eeq
We perform a substitution of $l=k-j$ on the left side of \eqref{eq:eq13} and then apply two Vandermonde identities \cite[pp. 9, 15]{Rio79}, 
\begin{align}
\sum_{k=0}^{n+1}\sum_{j=0}^m\binom{m}{j}\binom{k-j}{p-j}\binom{p}{k-j}\binom{n+1}{k+m-j} 
&=\sum_{l=0}^{n+1}\binom{p}{l}\binom{n+1}{m+l}\mathcol{\sum_{j=0}^m\binom{m}{j}\binom{l}{p-j}} \\
&=\sum_{l=0}^{n+1}\binom{p}{l}\binom{n+1}{m+l}\mathcol{\binom{m+l}{p}} \\ 
&=\sum_{j=0}^{n+1}\binom{p}{j-m}\binom{n+1}{j}\binom{j}{p} \\ 
&=\mathcol{\sum_{j=0}^{n+1}\binom{j}{p}\binom{p}{j-m}\binom{n+1}{j}} \\ 
&=\mathcol{\binom{n+1}{p}\binom{n+1}{m}}.  \nonumber\tag*{\qedhere}
\end{align}
\end{proof}

\begin{theorem}\label{thm:lagissums}
Let $T$ be a Laguerre diagonal differential operator, $T[L_n(x)]:=\gamma_n L_n(x)$, where $\{\gamma_n\}_{n=0}^\infty$ a sequence of real numbers. Then there is a sequence of polynomials, $\{Q_k(x)\}_{k=0}^\infty$, and a sequence of classical diagonal differential operators, $\{T_n\}_{n=0}^\infty$, such that 
\beq\nonumber
T[L_n(x)]:=\left(\sum_{k=0}^\infty Q_k(x)D^k\right) L_n(x) = \left(\sum_{k=0}^\infty T_kD^k\right) L_n(x) = \gamma_n L_n(x). 
\eeq
Then, for each $n\in\N_0$,
\beq\nonumber
\{b_{n,m}\}_{m=0}^\infty:=\left\{\sum_{k=0}^m \binom{m}{k} \frac{(-1)^n}{n!} \left(\sum_{j=0}^n \binom{n}{j} \frac{(k+j)!}{((k+j)-n)!}g_{k+j}^*(-1)\right)\right\}_{m=0}^\infty, 
\eeq
where $T_n[x^m]=b_{n,m}x^m$ for every $n,m\in\N_0$. 
\end{theorem}

\begin{proof}
The existence of the sequences $\{Q_k(x)\}_{k=0}^\infty$ and $\{T_k\}_{k=0}^\infty$ are established by Theorem \ref{thm:piotr} and \ref{thm:classic}. Recall from Theorem \ref{thm:classic} that, 
\beq
\{b_{n,m}\}_{m=0}^\infty=\left\{\sum_{k=0}^m \binom{m}{k} Q_{k+n}^{(k)}(0)\right\}_{m=0}^\infty. 
\eeq
Hence, we wish to verify that,   
\beq\label{form13}
Q_{k+n}^{(k)}(0)=\frac{(-1)^n}{n!} \left(\sum_{j=0}^n \binom{n}{j} \frac{(k+j)!}{((k+j)-n)!}g_{k+j}^*(-1)\right). 
\eeq
To ease the verification process, we first rewrite formula \eqref{form13} as follows, 
\beq\label{eq:lagjens}
Q_n^{(m)}(0)=\sum_{p=0}^{n} (-1)^{n-m} \binom{n-m}{p-m}\binom{p}{n-m}g_p^*(-1). 
\eeq
We will now verify formula \eqref{eq:lagjens}, \textit{tour de force}, by induction. Suppose for every $m\in\N_0$ and $k\in\{0,1,\ldots,n\}$, formula \eqref{eq:lagjens} holds for $Q_k^{(m)}(0)$. We now calculate $Q_{n+1}^{(m)}(0)$ using the recursive formula of Theorem \ref{thm:piotr}, equation \eqref{eq:jenreverse}, and Lemma \ref{lem:zerolag} and \ref{lm:horrible}, 
\begin{align}
Q_{n+1}^{(m)}(0) 
=&\frac{1}{L_{n+1}^{(n+1)}}\left(\mathcol{\gamma_{n+1}}\ \mathcol{L_{n+1}^{(m)}(0)} - \sum_{k=0}^n \frac{d^m}{dx^m}\left.\left[Q_k(x) L_{n+1}^{(k)}(x)\right]\right|_{x=0} \right) \\
=& (-1)^{n+1}\Bigg(\mathcol{\sum_{p=0}^{n+1} \binom{n+1}{p}g_p^*(-1)}\ \mathcol{\binom{n+1}{m}(-1)^m} \nonumber\\
 &\ \hspace{.2in}- \sum_{k=0}^n \sum_{j=0}^m \binom{m}{j} \mathcol{Q_k^{(j)}(0)}\ \mathcol{L_{n+1}^{(k+m-j)}(0)} \Bigg) \\
=& (-1)^{n+1}\Bigg(\sum_{p=0}^{n+1} \binom{n+1}{p}g_p^*(-1)\ \binom{n+1}{m}(-1)^m \nonumber\\
 &\ \hspace{.2in}-\sum_{k=0}^n \sum_{j=0}^m \binom{m}{j} \mathcol{\left(\sum_{p=0}^{n+1} \binom{k-j}{p-j}\binom{p}{k-j} (-1)^{k-j} g_p^*(-1)\right)} \nonumber\\ 
 &\ \hspace{1.7in}\mathcol{\left( \binom{n+1}{k+m-j}(-1)^{k+m-j}\right)} \Bigg) \\
=&\sum_{p=0}^{n+1} \Bigg( (-1)^{n+1-m}\Bigg(\binom{n+1}{p}\binom{n+1}{m} \nonumber\\
 &\ \hspace{.2in}-\sum_{k=0}^n\sum_{j=0}^m \binom{m}{j}\binom{k-j}{p-j}\binom{p}{k-j}\binom{n+1}{k+m-j}\Bigg)  \Bigg) g_p^*(-1) \\
=& \sum_{p=0}^{n+1} (-1)^{n+1-m}\binom{n+1-m}{p-m}\binom{p}{n+1-m} g_p^*(-1). \nonumber\tag*{\qedhere}
\end{align}
\end{proof}

Similar to the Hermite case (see Theorem \ref{thm:herm}) the following theorem establishes a Rodrigues type formula between $h_n(x)$ ($n\in\N_0$) and $f(x)$. This formula then relates the hyperbolicity preservation of $T$ with each $T_n$ ($n\in\N_0$). 

\begin{theorem}\label{thm:lag}
Suppose $\{\gamma_k\}_{k=0}^\infty$ is a non-trivial Laguerre multiplier sequence and let $\{g_k^*(x)\}_{k=0}^\infty$ be the reversed Jensen polynomials associated with $\{\gamma_k\}_{k=0}^\infty$. Then, for each $n\in\N_0$,
\beq\label{eq:crazylag}\nonumber
\{b_{n,m}\}_{m=0}^\infty:=\left\{\sum_{k=0}^m \binom{m}{k} \frac{(-1)^n}{n!} \left(\sum_{j=0}^n \binom{n}{j} \frac{(k+j)!}{((k+j)-n)!}g_{k+j}^*(-1)\right)\right\}_{m=0}^\infty, 
\eeq
is a Laguerre multiplier sequence. 
\end{theorem}

\begin{proof}
By assumption, $\{\gamma_k\}_{k=0}^\infty$ is a Laguerre multiplier sequence. Hence, by Theorem \ref{thm:mslag}, 
\beq\label{eq:lagfx}
f(x)=\sum_{k=0}^\infty \frac{f^{(k)}(0)}{k!} x^k := \sum_{k=0}^\infty g_k^*(-1) x^k \in \R[x]\cap\lp^s[-1,0]. 
\eeq
To show that, $\{b_{n,m}\}_{m=0}^\infty$ is a Laguerre multiplier sequence we must show that, 
\beq
h_n(x):=\sum_{m=0}^\infty \left(\sum_{k=0}^m \binom{m}{k} b_{n,k} (-1)^{m-k} \right) x^m \in \R[x]\cap\lp^s[-1,0]. 
\eeq
We use Theorem \ref{thm:comb} and perform the following calculations, 
\begin{align}
h_n(x)
&=\sum_{k=0}^\infty \left(\sum_{j=0}^k \binom{k}{j} b_{n,j} (-1)^{k-j} \right) x^k \\
&=\sum_{k=0}^\infty \left(\frac{(-1)^n}{n!} \left(\sum_{j=0}^n \binom{n}{j} \frac{(k+j)!}{((k+j)-n)!}g_{k+j}^*(-1)\right) \right)x^k \\
&=\frac{(-1)^n}{n!} \sum_{j=0}^n \binom{n}{j} \sum_{k=0}^\infty \left( \frac{(k+j)!}{((k+j)-n)!}g_{k+j}^*(-1)\right) x^k \\
&=\frac{(-1)^n}{n!} \sum_{j=0}^n \binom{n}{j} \sum_{k=0}^\infty \frac{f^{(k+j)}(0)}{((k+j)-n)!} x^k \\
&=\frac{(-1)^n}{n!} \sum_{j=0}^n \binom{n}{j}x^{n-j} D^n f(x) \\
&=\frac{(-1)^n}{n!} (1+x)^n D^n f(x). \label{eq:lagrecur}
\end{align}
Hence, if $f(x)\in\R[x]\cap\lp^s[-1,0]$, then $h_n(x)\in\R[x]\cap\lp^s[-1,0]$. 
\end{proof}

Similar to the Hermite case, equation \eqref{eq:lagrecur} also provides a recursive formula, 
\beq
h_n(x):=\frac{-1}{n}(x+1)^nD(x+1)^{1-n}h_{n-1}(x),\ \ \ \ \ (n\ge 1,\ h_0(x):=f(x)). 
\eeq
Thus, again, the hyperbolicity preservation of $T_n$ with a Laguerre multiplier sequence, is enough to establish that $T_{n+1}$ is hyperbolicity preserving with a Laguerre multiplier sequence. 

\begin{example}
We show, similar to Examples \ref{ex:notherm1} and \ref{ex:notherm2}, that it is possible for $T_n$ to be hyperbolicity preserving for every $n$ and yet $T$ fail to be hyperbolicity preserving. Consider the following non-Laguerre multiplier sequence (see \eqref{eq:lagfun} and Theorem \ref{thm:mslag}), 
\beq
\{a_n\}_{n=0}^\infty:=\{2,3,4,5,6,\ldots\}, 
\eeq
where
\beq
T[L_n(x)]:=a_nL_n(x). 
\eeq
From Theorem \ref{thm:classic}, we obtain $T=\sum_{n=0}^\infty T_n D^n$, where $T_n[x^m]=b_{n,m}x^m$ (see \eqref{eq:aform}) are classical diagonal differential operators. We calculate $f(x)$ from equation \eqref{eq:lagfx}, 
\beq\label{eq:lagfun}
f(x):=\sum_{k=0}^\infty g_k^*(-1)x^k = x+2. 
\eeq
Hence by formula \eqref{eq:lagrecur},
\begin{align*}
h_0(x)&=\sum_{k=0}^\infty \left(\sum_{j=0}^k \binom{k}{j} b_{0,j} (-1)^{k-j} \right) x^k=x+2, \\
h_1(x)&=\sum_{k=0}^\infty \left(\sum_{j=0}^k \binom{k}{j} b_{1,j} (-1)^{k-j} \right) x^k=-x-1,\ \ \ \ \ \text{and} \\
h_n(x)&=\sum_{k=0}^\infty \left(\sum_{j=0}^k \binom{k}{j} b_{n,j} (-1)^{k-j} \right) x^k=0,\ \ \ \ \ \text{for}\ n\ge 2.
\end{align*}
We see that $h_0(x)\not\in\R[x]\cap\lp^s[-1,0]$, hence $\{b_{0,k}\}_{k=0}^\infty$ is not a Laguerre multiplier sequence (see Theorem \ref{thm:mslag}). However, if we define a classical multiplier sequence, $W[x^m]:=\frac{1}{m!}x^m$, then 
\beq
\sum_{k=0}^\infty \frac{b_{n,k}}{k!}x^k = e^x W[h_n(x)]\in\lp^s. 
\eeq
Hence, $\{b_{0,k}\}_{k=0}^\infty$ is a classical multiplier sequence (see Theorem \ref{thm:ms}). In addition, $h_n(x)\in\R[x]\cap\lp^s[-1,0]$ for $n\ge 1$. Thus, each $T_n$ ($n\ge 0$) is hyperbolicity preserving (see Theorem \ref{thm:mslag}), each $T_n$ ($n\ge 1$) diagonalizes with a Laguerre multiplier sequence, but $T$ itself is not a hyperbolicity preserver. 
\end{example}

From the calculations of \eqref{eq:lagjens} we can also provide a formula for the $Q_k$'s in a Laguerre differential operator (cf. Theorem \ref{thm:hermqk} and \ref{thm:hermcomplex}). 

\begin{theorem}\label{thm:lagqk}
Let $\{\gamma_n\}_{n=0}^\infty$ be a sequence of real numbers and $\{Q_k(x)\}_{k=0}^\infty$ be a sequence of polynomials, such that, 
\beq
T[L_n(x)]:=\left(\sum_{k=0}^\infty Q_k(x)D^k\right)L_n(x)=\gamma_nL_n(x). 
\eeq
Then for each $n\in\N_0$, 
\beq\label{eq:lagqk}
Q_n(x)=\sum_{k=0}^{n}\left(\sum_{p=0}^{n} (-1)^{n-k} \binom{n-k}{p-k}\binom{p}{n-k}g_p^*(-1)\right)x^k, 
\eeq
where $\{g_k^*(x)\}_{k=0}^\infty$ are the associated reversed Jensen polynomials of $\{\gamma_n\}_{n=0}^\infty$. 
\end{theorem}

Similar to Theorem \ref{thm:hermcomplex}, we provide another formula for the $Q_k$'s in a Laguerre diagonal differential operator (cf. \cite[Proposition 216, p. 107]{Cha11}). 

\begin{theorem}\label{thm:lagqk2}
Let $\{\gamma_n\}_{n=0}^\infty$ be a sequence of real numbers and $\{Q_k(x)\}_{k=0}^\infty$ be a sequence of polynomials, such that, 
\beq
T[L_n(x)]:=\left(\sum_{k=0}^\infty Q_k(x)D^k\right)L_n(x)=\gamma_nL_n(x). 
\eeq
Then for each $n\in\N_0$, 
\beq\label{eq:lagqnn}
Q_n(x)=\sum_{k=0}^{n}\frac{(-x)^{k}}{k!}\sum_{j=0}^{n-k} \binom{n-k}{j} (-1)^j\gamma_{j}L_{j}(x). 
\eeq
\end{theorem}

\begin{proof}
The proof is very similar to the proof of Theorem \ref{thm:hermcomplex}. Define
\beq
\tilde{T}:=\sum_{n=0}^\infty Q_n(x)D^n, 
\eeq
where $Q_n(x)$ is defined from equation \eqref{eq:lagqnn}. We will establish the result by showing that $\tilde{T}[L_m(x)]=\gamma_m L_m(x)$ for every $m\in\N_0$. Define the evaluation operator, 
\beq
W:=\sum_{n=0}^\infty \frac{(-1)^n}{n!}x^nD^n. 
\eeq
Note that $W[f(x)]=f(0)$ for every polynomial $f(x)$. Using Theorem \ref{lem:zerolag} and formula $\binom{n}{k}\binom{k}{j} = \binom{n}{j}\binom{n-j}{k-j}$ (see \cite[p. 3]{Rio79}), we now evaluate $\tilde{T}$ at $L_m(x)$, 
\begin{align}
\tilde{T}[L_m(x)]
&= \sum_{n=0}^m\left(\sum_{k=0}^n \frac{(-x)^k}{k!} \sum_{j=0}^{n-k} \binom{n-k}{j}(-1)^j \gamma_j L_j(x)\right) L^{(n)}_m(x) \\
&= \sum_{j=0}^{m} (-1)^j \gamma_j L_j(x) \sum_{k=0}^m \sum_{n=0}^m \binom{n-k}{j}\frac{(-x)^k}{k!} L^{(n)}_m(x) \\
&= \sum_{j=0}^{m} (-1)^j \gamma_j L_j(x) \sum_{k=0}^m \sum_{n=0}^m \binom{k}{j}\frac{(-x)^{n-k}}{(n-k)!} L^{(n)}_m(x) \\
&= \sum_{j=0}^{m} (-1)^j \gamma_j L_j(x) \sum_{k=0}^m\binom{k}{j} \sum_{n=0}^m \frac{(-x)^{n}}{n!} L^{(n+k)}_m(x) \\
&= \sum_{j=0}^{m} (-1)^j \gamma_j L_j(x) \sum_{k=0}^m\binom{k}{j} W[ L^{(k)}_m(x)] \\
&= \sum_{j=0}^{m} (-1)^j \gamma_j L_j(x) \sum_{k=0}^m\binom{k}{j} L^{(k)}_m(0) \\
&= \sum_{j=0}^{m} (-1)^j \gamma_j L_j(x) \sum_{k=0}^m\binom{k}{j} \binom{m}{k}(-1)^k \\
&= \sum_{j=0}^{m} (-1)^j \gamma_j L_j(x) \sum_{k=0}^m\binom{m}{j} \binom{m-j}{k-j}(-1)^k \\
&= \sum_{j=0}^{m} (-1)^j \gamma_j L_j(x) \binom{m}{j}(-1)^j\sum_{k=0}^m\left(\binom{m-j}{k}(-1)^k\right) \\
&= \sum_{j=0}^{m} \binom{m}{j} \gamma_j L_j(x) \sum_{k=0}^m\binom{m-j}{k}(-1)^k \\
&= \gamma_m L_m(x) \nonumber\tag*{\qedhere}
\end{align}
\end{proof}

\vspace{.25in}
\section{Open Problems}

\begin{problem}
A frequent query in the literature is to find properties of the $Q_k$'s such that
\beq
T:=\sum_{k=0}^\infty Q_kD^k
\eeq
is hyperbolicity preserving. We ask instead a parallel question; what are the properties needed for classical diagonal differential operators, $T_n$'s, to form hyperbolicity preservers, as in 
\beq
T=\sum_{k=0}^\infty T_nD^n\ \ \ \text{?}
\eeq 
\end{problem}

\begin{problem}\label{que:que5}
Do the shifted Laguerre polynomials, $\{L_n(x-\alpha)\}_{n=0}^\infty$, possess the same property found in Theorem \ref{thm:lagissums} and Theorem \ref{thm:lag}? Generalized Laguerre? Generalized Hermite? 
\end{problem}

\begin{problem}
Find all hyperbolicity preservers that can be written as a sum of classical hyperbolicity preservers ($T=\sum_{k=0}^\infty T_kD^k$), as in Theorem \ref{thm:classic} or \ref{thm:classic2}. 
\end{problem}

\begin{problem}\label{pm:incdegtn}
Does there exist a hyperbolicity preserver of the form, 
\beq
T:=\sum_{k=-\infty}^\infty T_k D^k,
\eeq
such that $T_k\not\equiv 0$ for every $k\in\N$? Compare with the open problem on ``increasing degree'' of A. Piotrowski \cite[Problem 197, p. 172]{Pio07}.
\end{problem}

\begin{problem}
From T. Forg\'acs and A. Piotrowski \cite{FP13b} we are given an intriguing open problem. Namely, if 
\beq
T[H_n(x)]:=\left(\sum_{k=0}^\infty Q_k(x)D^k\right)H_n(x)=\gamma_n H_n(x)
\eeq
is a Hermite diagonal differential operator where $\{\gamma_n\}_{n=0}^\infty$ is a classic multiplier sequence and each $Q_k(x)$ has only real zeros, then can we conclude that $T$ is a hyperbolicity preserver (cf. Theorem \ref{thm:hermrealzeros})? 

Using the formulations throughout this paper, we pose two ideas that might prove of use to this question. First, following the method of T. Forg\'acs and A. Piotrowski we analysis the leading and second-leading coefficients of the $Q_k$'s (for the definition of $h_n$, see equation \eqref{h_n's}),  
\begin{align}
\frac{d}{dx}h_0(x)&=\sum_{k=0}^\infty \frac{Q_{k+1}^{(k+1)}(0)}{k!}x^k = e^{-x}(f'(x)-f(x)),\ \ \ \ \ \text{and} \\
(-4)\int h_1(x)dx&= (-4)\sum_{k=0}^\infty \frac{Q_{k+1}^{(k-1)}(0)}{k!}x^k =e^{-x}(f'(x)+f(x)),
\end{align}
where $f(x)=\sum_{k=0}^\infty \frac{\gamma_k}{k!}x^k$ and where we take $Q_1^{(-1)}(0):=f'(0)+f(0)$. Thus, in general we ask; if $\{\gamma_k\}_{k=0}^\infty$ is a non-increasing, positive, classic multiplier sequence, then can we conclude that $e^{-x}(f'(x)-f(x))$ and $e^{-x}(f'(x)+f(x))$ have at least one Taylor coefficient of opposite sign? 

Second, according to the Borcea-Branden Theorem \cite[Theorem 5]{BB09}, if
\beq
T:=\sum_{k=0}^\infty Q_k(x)D^k,\ \ \ \text{and,}\ \ \ W:=\sum_{k=0}^\infty \frac{Q_k(x)}{2^k}D^k, 
\eeq
then $T$ is hyperbolicity preserving if and only if $W$ is hyperbolicity preserving (see also the proof of Theorem \ref{thm:herminter}). However, if $T$ is also a Hermite diagonal differential operator, then only the hyperbolicity of $T[x^n]$ is needed to conclude that $T$ is a hyperbolicity preserver (see Theorem \ref{thm:minimal}). Can the same be said of $W$? This ``minimal set'' ($\{x^n\}_{n=0}^\infty$) that allows the conclusion of hyperbolicity preservation is a commonly sought after attribute of differential operators. We ask, what relationship do the sets $A$ and $B$ have, where
\begin{align}
A&=\left\{p(x):p(x)=\left(\sum_{k=0}^\infty Q_k(x)D^k\right)f_n(x)\right\},\ \ \ \text{and}\\
B&=\left\{p(x):p(x)=\left(\sum_{k=0}^\infty Q_k(x)\alpha^kD^k\right)f_n(x)\right\},
\end{align}
given $\{f_n(x)\}_{n=0}^\infty$ is some sequence of polynomials and $\alpha>0$? If $B$ only has hyperbolic polynomials, then must $A$ have only hyperbolic polynomials? What restrictions would allow this conditional to hold? 
\end{problem}

 %For one-column wide figures use
%\begin{figure}
 %Use the relevant command to insert your figure file.
 %For example, with the graphicx package use
  %\includegraphics{example.eps}
 %figure caption is below the figure
%\end{figure}
%
 %For two-column wide figures use
%\begin{figure*}
 %Use the relevant command to insert your figure file.
 %For example, with the graphicx package use
  %\includegraphics[width=0.75\textwidth]{example.eps}
 %figure caption is below the figure
%\end{figure*}
%
 %For tables use
%\begin{table}
 %table caption is above the table
 %For LaTeX tables use
%\begin{tabular}{lll}
%\hline\noalign{\smallskip}
%\end{tabular}
%\end{table}

\vspace{.25in}
\section{acknowledgements}
The author would like to thank Dr. Csordas' Spring 2014 seminar for the helpful discussions and many inspiring suggestions that led to this paper. In particular, thanks is given to Dr. Csordas for his endless hours of advice and instruction in refining this work.

\bibliography{ODofMS_BiB}{}
\bibliographystyle{spmpsci}

%\begin{thebibliography}{}
% etc
%\end{thebibliography}

\end{document}